\pdfoutput=1
\documentclass{article}

\usepackage{ifthen}
\usepackage{fullpage}
\usepackage{amsthm}
\usepackage{amsmath,amsfonts,amssymb,multirow}
\usepackage{algorithm}
\usepackage{algpseudocode} 
\usepackage{graphicx} 
\usepackage{csquotes} 
\usepackage{placeins} 
\usepackage[normalem]{ulem} 
\usepackage{xcolor}
\usepackage{subcaption}
\usepackage[noabbrev, capitalise]{cleveref}
\usepackage[noblocks]{authblk}

\usepackage{marvosym}

\usepackage{tikz}
\usetikzlibrary{arrows,calc}
\usepackage{url}
\usepackage[
backend=biber,
style=alphabetic,
sorting=ynt
]{biblatex}
\usepackage{color}

\usepackage{paralist} 
\usepackage{bbm} 
\usepackage{wrapfig} 
\usepackage{palatino} 
\usepackage{aligned-overset}

\newcommand{\mA}{\mathbf{A}}

\newcommand{\mI}{\mathbf{I}}

\newcommand{\mW}{\mathbf{W}}

\newcommand{\va}{\mathbf{a}}
\newcommand{\vb}{\mathbf{b}}

\newcommand{\ve}{\mathbf{e}}

\newcommand{\vr}{\mathbf{r}}
\newcommand{\vs}{\mathbf{s}}

\newcommand{\vv}{\mathbf{v}}

\newcommand{\vx}{\mathbf{x}}

\newcommand{\vz}{\mathbf{z}}

\newcommand{\Rmbb}{\mathbb{R}}

\providecommand{\oline}[1]{\mkern 1.5mu\overline{\mkern-1.5mu#1}}

\renewcommand{\hbar}{\oline{h}}

\newcommand{\norm}[1]{\left\| #1 \right\|}

\newcommand{\R}{\Rmbb}

\newtheorem{theorem}{Theorem}[section]
\newtheorem{lemma}[theorem]{Lemma}
\newtheorem{corollary}[theorem]{Corollary}

\newtheorem{definition}[theorem]{Definition}
\newtheorem{remark}[theorem]{Remark}
\newtheorem{prop}[theorem]{Proposition}

\newcommand{\xinit}{\vx_0}
\newcommand{\xk}{\vx_k}
\newcommand{\xkpo}{\vx_{k+1}}
\newcommand{\xopt}{\vx^\star}
\newcommand{\xkmo}{\vx_{k-1}}

\newcommand{\ek}{\ve_k}
\newcommand{\ekpo}{\ve_{k+1}}

\newcommand{\smin}{\sigma_{q-\beta, \mathrm{min}}^2}
\newcommand{\smax}{\sigma_{\mathrm{max}}^2}

\newcommand{\specificthanks}[1]{\@fnsymbol{#1}}

\let\svthefootnote\thefootnote
\newcommand\freefootnote[1]{
  \let\thefootnote\relax
  \footnotetext{#1}
  \let\thefootnote\svthefootnote
}

\graphicspath{{./figures/}{}}

\author[1]{Lu Cheng}
\author[1]{Benjamin Jarman}
\author[1]{Deanna Needell}
\author[2]{Elizaveta Rebrova}

\affil[1]{Department of Mathematics, University of California, Los Angeles}
\affil[2]{Department of Operations Research and Financial Engineering, Princeton University}

\begin{document}
\title{On Block Accelerations of Quantile Randomized Kaczmarz for Corrupted Systems of Linear Equations}

\date{\vspace{-5ex}}
\maketitle

\abstract{

 With the growth of large data as well as large-scale learning tasks, the need for efficient and robust linear system solvers is greater than ever. The randomized Kaczmarz method (RK) and similar stochastic iterative methods have received considerable recent attention due to their efficient implementation and memory footprint. 
These methods can tolerate streaming data, accessing only part of the data at a time, and can also approximate the least squares solution even if the system is affected by noise. However, when data is instead affected by large (possibly adversarial) corruptions, these methods fail to converge, as corrupted data points draw iterates far from the true solution. A recently proposed solution to this is the QuantileRK method, which avoids harmful corrupted data by exploring the space carefully as the method iterates. The exploration component requires the computation of quantiles of large samples from the system and is computationally much heavier than the subsequent iteration update. 

In this paper, we propose an approach that better uses the information obtained during exploration by incorporating an averaged version of the block Kaczmarz method. This significantly speeds up convergence, while still allowing for a constant fraction of the equations to be arbitrarily corrupted.  We provide theoretical convergence guarantees as well as experimental supporting evidence. 
We also demonstrate that the classical projection-based block Kaczmarz method cannot be robust to sparse adversarial corruptions, but rather the blocking has to be carried out by averaging one-dimensional projections.}
\section{Introduction}
\freefootnote{BJ and DN were partially supported by NSF DMS 2011140; BJ, DN, and ER were partially supported by NSF DMS 2108479.}

Let $\mA \in \mathbb{R}^{m \times n}$, $\vb \in \mathbb{R}^m$, and suppose we wish to find $\vx \in \mathbb{R}^n$ such that $\mA \vx = \vb$. Such linear systems are ubiquitous across applied mathematics and the sciences, arising in contexts ranging from medical imaging \cite{natterer, hounsfield_CAT} to machine learning \cite{bottou2010sgd}, sensor networks \cite{sensors}, and more. A common and widely studied approach is to seek the least squares solution $\vx_{\mathrm{LS}} = \operatorname{argmin}\norm{\mA \vx - \vb}$, for which many methods have been devised. 

In this paper, we consider the related problem of trying to solve a consistent system $\mA \vx = \vb^t$, where $\mA$ is of full rank, and whose solution is $\xopt$. Here, however, instead of observing the \emph{true} right hand side $\vb^t$ one observes a \emph{corrupted} version, $\vb = \vb^t + \vb^c$, where $\vb^c$ represents a vector of corruptions. In this setting, $\vx_{\mathrm{LS}}$ may be far from $\xopt$, rendering least squares solvers unsuitable. Frequently, such systems are highly overdetermined, with $m \gg n$, for example in settings where one has many more measurements than covariates. In this case, it is reasonable to hope to recover $\xopt$ as long as a sufficiently small fraction of rows are corrupted. Indeed, we assume that corruptions may be of arbitrary size and location, but affect only some fraction $\norm{\vb^c}_{0}/m := \beta \in [0,1)$ of data points. We refer to a row with a corrupted right hand side entry as a \emph{corrupted row}.

This model covers a wide variety of scenarios in which data may suffer corruptions during collection, transmission, storage, or otherwise. As one example, a frequent setting in which overdetermined linear systems appear is that of computerized tomography: in this case, each row of the system represents the absorption of a single X-ray beam through a medium, and solving the system recovers an image of said medium. A small number of beams malfunctioning may lead to catastrophic errors of arbitrary size in the resulting data, but as long as the number of such errors is a small one may still hope to recover the underlying solution to the uncorrupted system. Similar situations may arise in sensor networks from malfunctioning sensors, or error correcting codes from transmission errors. Note that typical methods for the least squares problem are unsuitable in this setting, as with arbitrarily large corruptions the least squares solution $\vx_{\mathrm{LS}}$ may be far from $\xopt$ (even if $\beta$ is small); this is contrary to the widely-studied \emph{noisy} setting, in which one assumes that every data point may be damaged by some small amount of noise, but that the least squares solution is still an accurate estimation of the solution.

This sparse corruption model is well-studied within the error-correction and compressed sensing literature: see \cite{candes2005decoding, eldar2012compressed, foucart2013mathematical}. However, such methods often require loading the entire system into memory; a requirement that is frequently impractical or impossible in settings where the system is large-scale, such as those systems arising in medical imaging applications \cite{hounsfield_CAT}. Recent works \cite{HadNeeCorr18, haddock2020quantilebased, Steinerberger2021QuantileBasedRK} have introduced novel approaches that in fact require only loading small portions (even single rows) of the system into memory at any time, whilst achieving linear convergence even in the presence of large -- or adversarially located -- corruptions.

In this work, we introduce a new iterative solver for corrupted linear systems, QuantileABK, building upon the averaged block Kaczmarz method introduced in \cite{Necoara2019FasterRB} and the quantile-based variant of randomized Kaczmarz, QuantileRK, introduced in \cite{haddock2021greed}. As with many iterative methods in the Kaczmarz family, QuantileABK relies on \emph{residual} information to determine the step size. The residual at the iterate $\vx_k$ is the vector of distances from $\vx_k$ to the hyperplanes defined by the rows of the matrix $\mA$, that is, $\vb - \mA \vx_k$. The standard randomized Kaczmarz method, on a consistent uncorrupted system, makes steps in the directions of projections to the individual hyperplanes of length equal to the corresponding residual entry. The underlying idea of the QuantileRK method is that large residual components suggest (a) potential corruptions and (b) large and potentially unstable next iteration steps. So, statistics of the absolute values of the residual entries are used to select trustworthy directions and only use them. We give more detailed backgrounds to each of the aforementioned prior methods in \cref{sec:background}. An important inefficiency of QuantileRK is that despite the entire residual being computed to detect corruptions, only a single row is used to compute the next iterate. Our method instead leverages the information gained from the residual with a more complex projection step to take a highly over-relaxed step size, leading to a huge acceleration in convergence over the single-row method QuantileRK \cite{haddock2021greed}. 

We prove several convergence results for the proposed method. For an example of the acceleration our method brings, here is a simplified restatement of one of the results that holds for a particular class of random matrices:

\begin{theorem}[Informal restatement of \cref{thm:subgcase}]\label{thm:informalsubgcase}
Assume that $\mA \in \R^{m \times n}$ satisfies a certain random matrix model (see \cref{def:subgausstype}) and has sufficiently large aspect ratio $m/n$. Suppose then that the system $\mA \vx = \vb$ has a fraction $\beta$ of corrupted rows, with $\beta$ sufficiently small. Then with high probability, the iterates produced by applying QuantileABK (see \cref{alg:QBRK}) to this system satisfy
\begin{equation*}
    \norm{\xk - \xopt}^2 \leq \left(1 - c\right)^k\norm{\xinit - \xopt}^2,
\end{equation*}
where $c$ depends only on a user-chosen quantile parameter (in particular, $c$ is independent of $m$ and $n$).
\end{theorem}
This result may be compared to (\cite{haddock2020quantilebased}, Theorem 1), to see that our method converges faster than QuantileRK in this setting by a factor linear in $n$, the number of columns of the system. Moreover, our method has a computational cost of the same order, with the most significant cost in both methods being the computation of the residual. We note that this acceleration in convergence occurs also in the uncorrupted case (i.e., when $\beta = 0$). See Section~\ref{sec:method} for the formal description of the algorithm and further discussion, and Section~\ref{sec:mainresults} for all theorem statements. Notably, we do not restrict ourselves to the random matrix setting: as in \cite{Steinerberger2021QuantileBasedRK}, we show a general guarantee of linear convergence, with a rate depending on the spectral properties of $\mA$ and its row submatrices (\cref{thm:mainthm}). 

The idea to leverage several equations to speed up Kaczmarz methods is not new, it is in the core of a sequence of Block Kaczmarz methods, including \cite{Elfving1980BlockiterativeMF, needell2013paved, needell2015blockls, Necoara2019FasterRB}. However, not all of them are equally extendable to the corrupted framework. The focus of this work is to discriminate between block Kaczmarz accelerations in terms of their provable robustness to adversarial corruptions: see additional discussion in Sections~\ref{s:block-k} and \ref{s:proj_vs_adv}.

\subsection{Organization}\label{sec:org}
The remainder of the paper is organized as follows. In \cref{sec:notation} we introduce notation used throughout the paper. In \cref{sec:background} we give a detailed background for previous methods upon which our method is built, and in \cref{sec:mainresults} we give a summary of our main results. \cref{sec:method} contains a description of our proposed method, and \cref{sec:theory} contains our theoretical results. In \cref{sec:experiments} we demonstrate our method in a range of experiments, and finally in \cref{sec:conclusion} we conclude and offer ideas for future directions.

\subsection{Notation}\label{sec:notation}
For a matrix $\mA \in \mathbb{R}^{m \times n}$, we denote its rows by $\va_i \in \mathbb{R}^{n}$, $i \in [m]$. For a collection of indices $\tau \subseteq [m]$, we let $\mA_{\tau}$ denote the matrix obtained from $\mA$ by restricting to rows indexed by $\tau$. We denote the operator norm of $\mA$ by $\norm{\mA}$, and the Fr\"obenius norm by $\norm{\mA}_F$. For a vector $\vv$ we denote its Euclidean norm by $\norm{\vv}$. For a matrix $\mA$ we denote its largest singular value by $\sigma_{\max}(\mA)$, and smallest by $\sigma_{\min}(\mA)$. When the matrix at hand is clear, we abbreviate these to $\sigma_{\max}$ and $\sigma_{\min}$.

In sections where we view $\mA$ as an instance of a certain family of random matrices, we use some definitions from probability. Namely, for a real-valued random variable $X$, we denote its subgaussian norm by $\norm{X}_{\Psi_2} := \inf\{t > 0 : \mathbb{E}(\exp(X^2/t^2)) \leq 2\}$. For a random vector $\vv \in \mathbb{R}^n$, its subgaussian norm is defined as $\norm{\vv}_{\Psi_2} := \sup_{\vx \in S^{n-1}} \norm{\langle \vv, \vx \rangle}_{\Psi_2}$. A random variable is said to be subgaussian if it has finite subgaussian norm. Lastly, a random vector $\vv \in \mathbb{R}^n$ is said to be isotropic if $\mathbb{E}(\vv \vv^\top) = \mI$, where $\mI$ denotes an appropriately-sized identity matrix.

We will frequently make use of a quantile of the absolute residual. For $q \in [0,1]$ and $\vx \in \mathbb{R}^n$, we denote the $q$\textsuperscript{th} quantile of the (corrupted) absolute residual $|\mA \vx - \vb|$ by
\[
Q_q(\vx) := \text{q\textsuperscript{th} quantile of }\{|\langle \va_i, \vx\rangle - b_i| : i \in [m]\},
\]
recalling that the $q$\textsuperscript{th} quantile of a multiset $S$ is the $\lceil qS\rceil$\textsuperscript{th} smallest element of $S$.

Lastly, we use $C, c, c_1, \cdots$ to denote absolute constants that may vary from line to line. Subscripts are used to denote dependence on particular quantities, e.g. $C_q$ denotes an absolute constant depending on $q$.

\subsection{Background \& Related Work}\label{sec:background}
\subsubsection{Randomized Kaczmarz}
The Kaczmarz method \cite{Kac37:Angenaeherte-Aufloesung} (later rediscovered for use in computerized tomography as the Algebraic Reconstruction Technique \cite{hounsfield_CAT}) is a popular iterative method for solving overdetermined consistent linear systems. An arbitrary initial iterate $\xinit$ is projected sequentially onto the hyperplanes corresponding to rows of the system $\mA \vx = \vb$, so that at the $k$\textsuperscript{th} iteration the update has the form 
\[
\xk = \xkmo - \frac{\va_i^\top \xkmo - b_i}{\norm{\va_i}^2}\va_i,
\]
where $i = k \text{ mod $m$}$. Whilst convergence to $\xinit$ is guaranteed via a simple application of Pythagoras's theorem, quantitative convergence guarantees proved elusive. In the landmark paper \cite{SV09:Randomized-Kaczmarz}, the authors proved a linear convergence guarantee when rows are selected at random according to a particular distribution. Namely, in their randomized Kaczmarz method, at iteration $k$ row $i$ is selected with probability $\norm{\va_i}^2/\norm{\mA}_F^2$, and the update takes the same form as above. This row selection scheme gave rise to \cref{thm:rkconv}.

\begin{theorem}[Strohmer \& Vershynin, 2007]\label{thm:rkconv}
Suppose that $\mA \vx = \vb$ is consistent with solution $\vx^\ast$. Then the iterates produced by applying randomized Kaczmarz to this system satisfy:
\[
\mathbb{E}\left(\norm{\xk - \xopt}^2\right) \leq \left( 1 - \frac{\sigma_{\mathrm{min}}^2}{\norm{\mA}_F^2}\right)^k \norm{\xinit - \xopt}^2.
\]
\end{theorem}

This result spurred a boom in related research, including Kaczmarz variants with differing row selection protocols \cite{Steinerberger2021AWR, Haddock2018OnMM, bai2018greedy}, block update methods \cite{Elfving1980BlockiterativeMF, popa1997, needell2013paved}, and adaptive methods \cite{gower2021adaptive}. Our method is motivated by block methods in particular, which we proceed to discuss in more detail. 

\subsubsection{Block Kaczmarz Methods}\label{s:block-k}

Variants of the Kaczmarz method that make use of more than a single row at each iteration, often referred to as block methods, have been extensively studied. Two particular methodologies have proven popular: 
\begin{itemize}
    \item \emph{projective} block methods, in which at each iteration the iterate is projected onto the subspace defined by an entire block of rows \cite{needell2013paved, popa1997, Elfving1980BlockiterativeMF}, and
    \item \emph{averaged} block methods, in which at each iteration the projections of the previous iterate onto each individual row in a block are computed and then averaged \cite{Necoara2019FasterRB, moorman2021average}.
\end{itemize}

Consider first the projective methodology. It has been shown the projective block Kaczmarz significantly outperforms randomized Kaczmarz \cite{needell2013paved}, particularly in the case when the system has coherent rows \cite{needell2013two}. Each iteration of the projective algorithm computes the best possible update given the information from the considered block, however, block projections are known to be significantly less stable for more sophisticated tasks, for example linear feasibility problems \cite{briskman2015block}. 

The presence of corruptions may also significantly disrupt projective block variants. Whilst in the single-row setting the quantile statistic is able to control the potential harm caused by projecting onto a corrupted row, a block containing a corrupted row may yield a projection that is arbitrarily far from the true solution. To some extent, this issue can be alleviated by posing an assumption of row \emph{incoherence}: that every two rows are not nearly parallel, i.e., their normal vectors have small scalar products. Informally, this results in the intersection subspaces being ``close enough" to individual projection points due to non-trivial angles between the solution hyperplanes for individual equations. The incoherence condition is implicitly needed in the existing non-block QuantileRK results \cite{haddock2020quantilebased, Steinerberger2021QuantileBasedRK} to ensure that the quantile statistic is representative. In this work, it also appears in the form of a restricted smallest singular value, discussed below.

However, the incoherence assumption does not resolve the second deficiency of projective block methods applied to corrupted systems. Namely, a residual-based criterion for deciding if a certain equation is trustworthy or corrupted cannot guarantee to identify all corrupted equations: for example, a current iterate might satisfy some corrupted equation exactly. Projecting onto a block containing a corrupted equation keeps the iterate inside its corrupted (shifted) hyperplane. Finally, when increasing block size, one rapidly increases the chance of an adversarial setting in which the majority of the blocks contain at least one corrupted row. A concrete adversarial construction for projective block methods is discussed in \cref{s:proj_vs_adv}.

Given the lack of robustness of projective block methods, we focus in this work on modifying an averaged block Kaczmarz method introduced by Necoara \cite{Necoara2019FasterRB} and also considered in \cite{moorman2021average}. For a consistent system $\mA \vx = \vb$ with solution $\xopt$, at the $k$\textsuperscript{th} iteration a block of row indices $\tau_k$ is selected from a distribution $\mathcal{D}$ on $[m]$. Then, the projections of $\xkmo$ onto each row in $\tau$ are computed and averaged, possibly in a weighted fashion. A step of size $\alpha_k$ -- potentially dependent on the iteration -- is then taken in this averaged direction. The update is thus given by
\[
\xk = \xkmo - \alpha_k \sum_{i \in \tau_k}w^k_i\frac{\va_i^\top \xkmo - b_i}{\norm{\va_i}^2}\va_i, \quad \text{ with weights $w^k_i$ such that } \sum_{i \in \tau_k}w^k_i = 1.
\]
The method may be found in full as Algorithm 4.1 in \cite{Necoara2019FasterRB}, and we refer to it as AveragedRBK. The convergence of AveragedRBK depends on the spectra of the row submatrices formed by sampled blocks. Indeed, the key quantity
\begin{equation*}
\sigma^2_{\mathcal{D}, \max} := \max_{\tau \sim \mathcal{D}} \sigma^2_{\max}(\mA_\tau),
\end{equation*}
the largest singular value of any row-submatrix with rows sampled from $\mathcal{D}$.

Necoara's framework allows many freedoms: in row selection strategy, weighting scheme, and step size. Specializing to the particular case of uniformly weighted rows, a constant (optimized) step size, and fixed block size (but without restraint on other aspects of $\mathcal{D}$), the following convergence result holds.

\begin{theorem}[Necoara, 2019]\label{thm:avebrkconv}
Suppose that the system $\mA \vx = \vb$ is consistent with solution $\xopt$, and that $\mA$ has been normalized such that each row has unit norm. Then the iterates produced by applying AveragedRBK with block size $|\tau|$, step size $\frac{|\tau|}{\sigma^2_{\mathcal{D}, \max}}$, and row weights $1/|\tau|$, satisfy
\[
\mathbb{E}\left(\norm{\xk - \xopt}^2\right) \leq \left(1 - \frac{|\tau|\sigma_{\mathrm{min}}^2}{m\sigma^2_{\mathcal{D}, \max}}\right)^k \norm{\xinit - \xopt}^2.
\]
\end{theorem}

We include this particular result as it allows for easier comparison with other methods, but we refer the reader to \cite{Necoara2019FasterRB} for more general results. In particular, we see that under the setup of \cref{thm:avebrkconv}, AveragedRBK achieves an improvement in convergence rate by a factor of $|\tau|/\sigma^{2}_{\mathcal{D}, \max}$ compared to RK (recall \cref{thm:rkconv}). This is greater than one in most sensible cases, for instance if $\operatorname{rank}(\mA_{\tau}) \geq 2$ for all $\tau$, and will represent a significant speedup in cases where the sampled blocks are well-conditioned. We refer to Section 4.3 of \cite{Necoara2019FasterRB} for further details. Furthermore, we note that the accelerated convergence rate does not necessarily come with greater computation time as the individual row projections may be performed in parallel: see \cite{moorman2021average}.

\subsubsection{Kaczmarz Variants for Least Squares}

Research on randomized Kaczmarz and its variants originated in the setting of a consistent, full rank system. Since then, convergence results have been extended to the rank-deficient (but still consistent) case for randomized Kaczmarz \cite{zouzias2013rek} and projective block Kaczmarz \cite{Haddock2021PavingTW}. Generalizing results and methods to the inconsistent setting has also been an area of interest, for example in \cite{needell2010randomized} the author shows the following result in the setting of a \emph{noisy} system.

\begin{theorem}[Needell, 2010]
Suppose that $\mA\vx = \vb$ is a consistent system with solution $\xopt$, and that $\vr$ is some vector of noise. Then the iterates produced by applying randomized Kaczmarz to the system $\mA \vx = \vb + \vr$ satisfy
\[
\mathbb{E}\left(\norm{\xk - \xopt}^2\right) \leq \left(1 - \frac{\sigma_{\mathrm{min}}^2}{\norm{A}_F^2}\right)^k \norm{\xinit - \xopt}^2 + \frac{n}{\sigma_{\mathrm{min}}^2}\norm{\vr}_\infty^2.
\]
\end{theorem}

This result shows that randomized Kaczmarz is guaranteed to converge at the same rate as for a consistent system, but only up to some error horizon. Similar results, of convergence to a horizon, have been shown for projective block Kaczmarz \cite{needell2013paved}, averaged block Kaczmarz \cite{moorman2021average}, and other variants \cite{Haddock2018OnMM}.

Other works have developed methods that converge all the way to the least squares solution \cite{censor1983inconsistent, zouzias2013rek}. For example, in randomized extended Kaczmarz \cite{zouzias2013rek}, randomized Kaczmarz is applied simultaneously to the systems $\mA^\top \vz = 0$ and $\mA \vx = \vb - \vz$, with the $\vz$ and $\vx$ iterates converging to $\vb_{\operatorname{Im}(\mA)^\perp}$ and $\vx_{\mathrm{LS}}$ respectively. Recent works have expanded this idea to both projective and averaged block variants \cite{needell2015blockls, du2020reabk}. However, as noted previously, such methods are unsuitable in the sparse corruption model as $\vx_{\mathrm{LS}}$ may be a poor approximation of the true solution $\vx^\ast$. We discuss previous works in this direction next.

\subsubsection{Quantile Randomized Kaczmarz}
The first study of Kaczmarz methods for the sparse corruption model may be found in \cite{HadNeeCorr18}, in which the authors make use of the notion that corrupted rows are likely to have larger residual entries, as their corresponding hyperplane is displaced far from both the current iterate and true solution. Through applying several rounds of Kaczmarz-type iterations, such corrupted rows may be detected with high probability. However, the method requires severe restrictions on the number of corrupted rows. In particular, the method does not support the sparse corruption model we consider here, in which the number of corruptions scales linearly with the number of rows.

In \cite{haddock2020quantilebased}, the authors expand on this residual-based heuristic and introduce a quantile-based modification of randomized Kaczmarz, QuantileRK, which also attempts to detect and avoid projecting onto corrupted rows. A sample of rows is taken and a quantile of the resultant (absolute) subresidual is computed, and then one further row is sampled. If this sampled row has absolute residual entry below the quantile, it is deemed acceptable for projection, otherwise the iterate remains unchanged. The algorithm is given in full in \cite{haddock2020quantilebased} as Method 1.

Whilst extensive experiments in \cite{haddock2020quantilebased} indicate the effectiveness of QuantileRK for a variety of systems, corruption models, and very high corruption rates (values of $\beta$ up to $0.5$), the authors require significant restrictions on the matrix $\mA$ for their theoretical results. In particular, they assume a random matrix heuristic, captured in the following definition.

\begin{definition}\label{def:subgausstype}(Subgaussian-type systems)
Let $\mA \in \mathbb{R}^{m \times n}$ be a random matrix. We say that $\mA$ is of subgaussian-type if all of the following hold:
\begin{enumerate}
    \item $\norm{\va_i} = 1$ for all $i \in [m]$.
    \item $\sqrt{n}\va_i$ is mean-zero and isotropic for all $i \in [m]$.
    \item For some $K > 0$, $\norm{\sqrt{n}\va_i}_{\Psi_2} \leq K$ for all $i \in [m]$.
    \item For some $D > 0$, every entry $a_{ij}$ of $\mA$ has density function uniformly bounded by $D\sqrt{n}$.
\end{enumerate}
Here, $K$ and $D$ are absolute constants independent form the size of the matrix.
\end{definition}

These conditions are satisfied, for example, by a matrix whose rows are sampled uniformly from the unit sphere in $\mathbb{R}^n$. With these constraints, in \cite{haddock2020quantilebased} the authors prove the following high-probability linear convergence guarantee, without placing any restriction on the size or placement of corruptions but with an additional requirement that $\mA$ be sufficiently tall.

\begin{theorem}[Haddock et al., 2021]\label{thm:qrkconv}
Assume that $\mA$ is of subgaussian-type, and that the system $\mA \vx = \vb$ has a fraction $\beta$ of corrupted rows. Then with high probability, the iterates produced by QuantileRK with $t = m$ (i.e., the full residual is computed at each iteration) applied to this system satisfy
\[
\mathbb{E}\left(\norm{\xk - \xopt}^2\right) \leq \left( 1- \frac{C_q}{n}\right)\norm{\xinit - \xopt}^2,
\]
for some constant $C_q$, so long as $\beta \leq \min(c_1q^2, 1-q)$ and $m \geq Cn$.
\end{theorem}

In \cite{Steinerberger2021QuantileBasedRK}, Steinerberger sought to generalize the theory behind QuantileRK beyond the random matrix setting. Indeed, he distilled the critical controls that the random matrix heuristic provides to conditions on the quantity
\begin{equation}\label{smin_def}
\smin := \inf_{\tau \subset [m], |\tau| = (q - \beta)m}\sigma_{\min}^2(\mA_{\tau}).
\end{equation}
Whilst assuming that $\mA$ is of subgaussian-type allows for estimations of $\smin$ (see \cite{haddock2020quantilebased}, Proposition 1), one may also give a much more general convergence result in terms of this quantity, albeit with stricter relative conditions on $q$ and $\beta$.

\begin{theorem}[Steinerberger, 2021]\label{thm:steiner}
Suppose that $\mA \vx = \vb$ has a fraction $\beta$ of corrupted rows. Then for $\beta < q < 1-\beta$, if
\[
\frac{q}{q-\beta}\left(\frac{2\sqrt{\beta}}{\sqrt{1- q - \beta}} + \frac{\beta}{1- q- \beta}\right) < \frac{\smin}{\smax},
\]
then the iterates of QuantileRK$(q)$ with $t = m$ applied to this system satisfy
\[
\mathbb{E}\left(\norm{\xk - \xopt}^2\right) \leq (1 - c_{\mA, \beta, q})^k\norm{\xinit - \xopt}^2,
\]
where
\[
c_{\mA, \beta, q} = (q-\beta)\frac{\smin}{q^2 m} - \frac{\smax}{qm}\left(\frac{2\sqrt{\beta}}{\sqrt{1- q - \beta}} + \frac{\beta}{1- q- \beta}\right) > 0.
\]
\end{theorem}

Informally, the convergence rate is good if the uniform restricted smallest singular value $\smin$ is well-separated from zero, which itself may be viewed as a version of the incoherence assumption mentioned above in Section~\ref{s:block-k}. Indeed, a row subsystem $\mA_\tau$ with nearly parallel rows is nearly degenerate and $\sigma_{\mathrm{min}}({\bf A}_{\tau})$ is very small. On the other hand, independent subgaussian rows are nearly mutually orthogonal with high probability (see, e.g., \cite{vershynin_2018}) and have $\sigma_{\mathrm{min}}({\bf A}_{\tau}) = O(\tau/n)$ when $\tau \gg n$. Further discussion in \cite{Steinerberger2021QuantileBasedRK} aids in understanding the relative condition on $q, \beta$, and $\smin$ of \cref{thm:steiner}, as well as drawing connections to the random matrix case studied in \cite{haddock2020quantilebased}. 

\subsection{Summary of Main Results}\label{sec:mainresults}
We introduce a new method, quantile averaged block Kaczmarz (QuantileABK), that applies the quantile-based techniques of QuantileRK to the averaged block Kaczmarz method. Namely, at each iteration a sample of rows is taken, the quantile of the corresponding subresidual is computed, and then an iteration of averaged block Kaczmarz is performed using \emph{every} row with residual entry below the quantile. We defer a full explanation of the method to \cref{sec:method}, including discussions on appropriate weights and step sizes.

\cref{thm:mainthm} shows that our method is guaranteed to converge at least linearly as long as $q, \beta$ and $\smin$ satisfy a similar constraint to that in \cref{thm:steiner}, without any assumption of randomness on $\mA$ (but still upholding the assumptions of full rank and unit norm rows). The proof of Theorem~\ref{thm:mainthm} can be found in Section~\ref{sec:proof_main}. 
\begin{theorem}\label{thm:mainthm}
Let $A \in \mathbb{R}^{m \times n}$ be of full rank with unit-norm rows. Suppose that the system $\mA \vx = \vb$ has a fraction $\beta$ of corrupted entries, and that $\beta < q < 1 - \beta$. If
\[
\frac{\sqrt{\beta}}{\sqrt{1 - q - \beta}} < \frac{\smin}{\smax},
\]
then the iterates of QuantileABK$(q)$, using a theoretically optimal step size, applied to this system satisfy

\begin{equation*}
    \norm{\xk - \xopt}^2 \leq \left(1 - \frac{c_1^2}{4c_2}\right)^k\norm{\xinit - \xopt}^2,
\end{equation*}
where
\[
c_1 = \frac{2 \smin}{qm} - \frac{2 \sqrt{\beta}\smax}{qm\sqrt{1-q-\beta}}, \qquad c_2 = \frac{\smax\smin}{q^2 m^2} - \frac{2\sqrt{\beta}\smax\smin}{q^2 m^2 \sqrt{1-q-\beta}} + \frac{\beta \sigma_{\mathrm{max}}^4}{q^2 m(1-q-\beta)}.
\]
\end{theorem}

The constants $c_1, c_2$ are difficult to interpret, so we include a different viewpoint in \cref{cor:mainthmcor} (also proved in Section~\ref{sec:proof_main}) to give a better idea of scaling. 

\begin{corollary}\label{cor:mainthmcor}
If in \cref{thm:mainthm} we choose $q$ such that for some $\epsilon \in [0,1)$
\[
\frac{\sqrt{\beta}}{\sqrt{1 - q - \beta}} = \epsilon\frac{\smin}{\smax},
\]
then the optimal step size may be expressed as
\begin{equation}\label{alpha}
\alpha_{opt} = \frac{qm(1-\epsilon)}{\smax - \epsilon(2-\epsilon)\smin},
\end{equation}
and we have the following convergence guarantee:
\[
\norm{\xk - \xopt}^2 \leq \left(1 - \frac{(1-\epsilon)^2 \smin}{\smax - \epsilon(2 - \epsilon)\smin}\right)^k \norm{\xinit - \xopt}^2.
\]
\end{corollary}

In general, the quantity $\smin$ is hard to estimate - both theoretically and empirically, particularly for very tall $\mA$. By specializing to the case of $\mA$ being of subgaussian-type (recall \cref{def:subgausstype}) we utilize results from \cite{haddock2020quantilebased} to estimate $\smin$ and obtain \cref{thm:subgcase}, a formal statement of the earlier \cref{thm:informalsubgcase}. The proof of Theorem~\ref{thm:subgcase} can be found in Section~\ref{sec:proof_subgaus}. 

\begin{theorem}\label{thm:subgcase}
Let $A$ be a random matrix satisfying \cref{def:subgausstype} with constants $K$ and $D$. Suppose then that the system $\mA \vx = \vb$ has a fraction $\beta$ of corrupted entries, with $\beta < q < 1-\beta$ and we have
\begin{equation}\label{eq:betacondition}
0 \le \epsilon < 1, \quad  \text{ where } \epsilon:= \frac{\beta}{C_3(1 - q + \beta)(q- \beta)^6},
\end{equation} 
and $C_3$ is an absolute constant depending only on the distribution of the rows of $\mA$. Suppose furthermore that $\mA$ has sufficiently large aspect ratio,
\[
\frac{m}{n} > C_4 \frac{1}{q-\beta}\log \frac{DK}{q-\beta}.
\]
Then the optimal step size for QuantileABK$(q)$ is 
\begin{equation}\label{eq:subgalpha}
\alpha_{opt} = c_{\epsilon, q, \beta}n,
\end{equation}
where $c_{\epsilon, q, \beta}$ is a constant depending on $\epsilon, q, \beta$. Moreover, with probability at least $1 - c\exp(-c_q m)$ the iterates of QuantileABK$(q)$, using the step size given in \cref{eq:subgalpha}, satisfy
\begin{equation*}
    \norm{\xk - \xopt}^2 \leq \left(1 - C_q\right)^k\norm{\xinit - \xopt}^2,
\end{equation*}
where $c_{q}, C_q$ depend only on $q$ (in particular, they are independent of $m$ and $n$).
\end{theorem}

As a concrete example, if $\beta = 0.012$ and $\mA$ is a sufficiently tall normalized Gaussian matrix, the conditions of \cref{thm:subgcase} allow taking $q$ as large as  $0.8486$. With $\beta = 0.012$, $q = 0.8486$, we obtain a convergence rate of $C_q = C_{0.8486} \geq 0.0287$ (note that this is independent of the size of the system and decreases initial distance to the solution $10$ times in $80$ iterations).

Finally, in all the theorems, one does not have to compute the optimal step size precisely to get convergence rate of the optimal order. In particular, Remark~\ref{approx_rate} shows that with the step size $\tilde\alpha = \xi \alpha_{opt}$ with $\xi \in (0,2)$,  the convergence rate is $(\xi - \xi^2/2)$ times the ``optimal" convergence rate.

\section{Proposed Method}\label{sec:method}

Here we provide a formal description of our algorithm. Under the same heuristic as in \cite{haddock2020quantilebased, Steinerberger2021QuantileBasedRK}, we use the $q$-quantile of the absolute residual $|\mA\vx - \vb|$ as a threshold to detect and avoid projecting onto rows that are too far from the current iterate (and thus, are likely to be corrupted). Then, an iteration of averaged block Kaczmarz is performed using the $qm$ rows with residual entries less than the computed quantile, using a fixed step size $\alpha$.
\begin{center}
\begin{minipage}{.7\textwidth}
\begin{algorithm}[H]
	\caption{Quantile Averaged Block Kaczmarz}\label{alg:QBRK}
	\begin{algorithmic}[1]
		\Procedure{QuantileABK}{$\mA,\vb$, $N$, $q$, $\alpha$, $\xinit$}
		\For{$k = 1, 2, \ldots, N-1$}
		    \State Compute $Q_q(\vx_{k-1}) = q\textsuperscript{th}\text{ quantile of }\{|\va_i^T \vx_{k-1} - b_i| : i \in [m]\}$
		    \State Set $\tau = \{i \in [m] : |\va_i^T \vx_{k-1} - b_i| < Q_q(\vx_{k-1})\}$
    		\State Update $\xk = \vx_{k-1} - \frac{\alpha}{|\tau|}\sum_{i \in \tau}(\va_i^T \vx_{k-1} - b_i)\va_i$
		\EndFor
		\State \Return $\vx_N$
		\EndProcedure
	\end{algorithmic}
\end{algorithm}
\end{minipage}
\end{center}
Note here that we show the algorithm as running for a prespecified number of iterations $N$, but in practice one may use any desired stopping criterion.

We note that the iterates of both QuantileRK and QuantileABK are at least $O(\beta m)$-times more computationally intensive than the standard RK method. This is because one needs to compute the residual entries of that many rows to obtain a quantile statistic that is able to accurately detect corrupted rows.

We note that the performance of the method depends heavily on the parameters $q$ and $\alpha$. In \Cref{sec:theory}, we prove our main convergence result, including a derivation of an optimal value of $\alpha$ and constraints on $q$ to ensure convergence. We follow this with experiments in \Cref{sec:experiments} to examine the optimal choice of $\alpha$ in practice, and to show the effects of varying $q$.

\begin{remark}
Note that in \cite{Necoara2019FasterRB}, a \textit{weighted} average is taken at each iteration, whilst we take an unweighted average. We reason that in our method, there is no particular reason to weight some rows more heavily than others: whilst one may be inclined to weight rows, say, proportionally to their residual entry, this has the knock-on effect of weighting potentially corrupted rows more heavily. However, we believe that our analysis may be extended to include additional weight parameters. 
\end{remark}

\begin{remark}
We choose to use a fixed step size at each iteration, but it is possible to extend the method to have varying step size. In particular, the theoretically optimal step size derived in \cref{thm:mainthm} is difficult to estimate \textit{a priori}, and may be substituted with an adaptive step size calculated only with information available at runtime as analyzed in \cite{Necoara2019FasterRB}. Note that the QuantileSGD method proposed in \cite{haddock2020quantilebased}, like QuantileRK, also utilizes the idea of varying the step size. While QuantileRK uses the quantile of the residual to decide whether to update the current iterate, QuantileSGD always does the weighted update, with the step size determined by the quantile size (and thus changing with the iterations).  
\end{remark}

\section{Theoretical Results}
\label{sec:theory}
\subsection{Preliminaries}

We begin our theory by introducing requisite preliminary results from \cite{Steinerberger2021QuantileBasedRK, haddock2020quantilebased}, and we include their proofs for completeness. Firstly, we provide an estimate on the residual quantiles computed at each iteration. Such an estimate is necessary to bound the impact that corrupted rows passing under this threshold can have on convergence. This is (\cite{Steinerberger2021QuantileBasedRK}, Lemma~1) and is a deterministic version of (\cite{haddock2020quantilebased}, Corollary~1). 

\begin{lemma}\label{lem:quantileestimate}
Consider applying QuantileABK to the system $\mA \vx = \vb$. If $0 < q < 1 - \beta$, then the quantile computed at the $k$-th iteration satisfies
\begin{equation*}
    Q_q(\xk) \leq \frac{\sigma_{\mathrm{max}}}{\sqrt{m}\sqrt{1-q-\beta}}\norm{\xk - \xopt}.
\end{equation*}
\end{lemma}
\begin{proof}
We follow \cite{Steinerberger2021QuantileBasedRK}. Let $\tau_1, \tau_2 \subset [m]$ denote the sets of indices of uncorrupted and corrupted rows respectively. Note that $|\tau_2| \leq \beta m$. We then have that
\begin{align*}
    \sum_{i \in \tau_1}|\langle \va_i, \xk \rangle - b_i|^2 &= \norm{\mA_{\tau_1}\xk - \vb_{\tau_1}}^2 \\
    &= \norm{\mA_{\tau_1}\xk - \mA_{\tau_1}\xopt}^2 \\
    &\leq \norm{\mA_{\tau_1}}^2\norm{\xk - \xopt}^2 \\
    &\leq \norm{\mA}^2\norm{\xk - \xopt}^2 \\
    &= \sigma_{\mathrm{max}}^2\norm{\xk - \xopt}^2.
\end{align*}

Next, note that by the definition of $Q_q(\xk)$, at least $(1-q)m$ rows have absolute residual entry greater than $Q_q(\xk)$, and at least $(1-q-\beta)m$ of those are uncorrupted. Therefore,
\begin{equation*}
    m(1-q-\beta)Q_q(\xk)^2 \leq \sum_{i \in \tau_1}|\langle \va_i, \xk\rangle - b_i|^2 \leq \sigma_{\mathrm{max}}^2\norm{\xk - \xopt}^2.
\end{equation*}
Rearranging gives the result.
\end{proof}

Next, we provide an estimate on the coherence of any subset of rows of $\mA$ of fixed size. This is necessary to control the adversarial case in which corruptions occur on coherent rows. We replicate (\cite{haddock2020quantilebased}, Lemma 4), but without randomness assumptions on $\mA$.

\begin{lemma}\label{lem:innerprodestimate}
Let $\mA \in \mathbb{R}^{m \times n}$ and let $\vx \in \mathbb{R}^n$. Then for every set of row indices $\tau \subseteq [m]$, we have
\[
\sum_{i \in \tau}|\langle \vx, \va_i \rangle | \leq \sigma_{\mathrm{max}} \sqrt{|\tau|}\norm{\vx}.
\]
\end{lemma}
\begin{proof}
As in \cite{haddock2020quantilebased}, let $\vs \in \mathbb{R}^m$ have entries
\[
s_i = \begin{cases} \operatorname{sign}(\langle \vx, \va_i \rangle), & \text{ if $i \in \tau$} \\
0, &\text{ otherwise,} \end{cases}
\]
for $i \in [m]$. Then we have
\[
\sum_{i \in \tau} |\langle \vx, \va_i \rangle | = \sum_{i = 1}^m \langle \vx, s_i \va_i \rangle = \langle \vx, \sum_{i = 1}^m s_i \va_i\rangle \leq \norm{\sum_{i = 1}^m s_i \va_i}\norm{\vx} = \norm{\mA^\top \vs}\norm{\vx} \leq \sigma_{\mathrm{max}}\sqrt{|\tau|}\norm{\vx},
\]
as desired.
\end{proof}

Lastly, we give a bound on the norm of the sum of a block of rows of $\mA$. This will be used in conjunction with \cref{lem:quantileestimate} to bound the disruptive effects of corrupted rows that pass under the quantile threshold. We note that to the best of our knowledge, \cref{lem:normsum} is a new result.

\begin{lemma}\label{lem:normsum}
Let $\mA \in \mathbb{R}^{m \times n}$. Then for every set of row indices $\tau \subseteq [m]$, we have
\[
\norm{\sum_{i \in \tau} \va_i}^2 \leq \smax |\tau|.
\]
\end{lemma}
\begin{proof}
We use a similar trick to \cref{lem:innerprodestimate}. Namely, let $\vs \in \mathbb{R}^m$ have entries
\[
s_i = \begin{cases} 1, & \text{ if $i \in \tau$} \\ 0 & \text{ otherwise,}\end{cases}
\]
for $i \in [m]$. Note that $\norm{\vs}^2 = |\tau|$. Then we have that
\[
\norm{\sum_{i \in \tau} \va_i}^2 = \norm{\mA^T \vs}^2 \leq \norm{\mA^T}^2 \norm{\vs}^2 = \smax |\tau|,
\]
as claimed.
\end{proof}
\subsection{General Case}\label{sec:proof_main}
Armed with our lemmas from the previous section, we are now ready to prove \cref{thm:mainthm}.
\begin{proof}[Proof of \cref{thm:mainthm}]
Denote by $\tau$ the set of row indices passing the quantile test at iteration $k+1$, i.e. 
\[
\tau = \{i \in [m] : |b_i - \va_i^\top \xk| \leq Q_q(\xk)\}.
\]
Denote by $\tau_1, \tau_2 \subset \tau$ the subsets of indices corresponding to uncorrupted and corrupted rows respectively. Note that we have $|\tau| = qm$, $|\tau_1| \geq (q-\beta)m$, $|\tau_2| \leq \beta m$. As is typical in Kaczmarz-esque convergence proofs, we now attempt to bound $\norm{\xkpo - \xopt}^2$ in terms of $\norm{\xk - \xopt}^2$. We have:
\begin{align*}
    \norm{\xkpo - \xopt}^2 &= \norm{\xk - \frac{\alpha}{|\tau|}\sum_{i \in \tau}(\va_i^T \vx_{k} - b_i)\va_i - \xopt}^2 \\
    &= \norm{\xk - \frac{\alpha}{|\tau|}\sum_{i \in \tau_1}(\va_i^T \vx_{k} - b_i)\va_i - \frac{\alpha}{|\tau|}\sum_{i \in \tau_2}(\va_i^T \vx_{k} - b_i)\va_i - \xopt}^2 \\
    &= \norm{(\xk - \xopt) - \frac{\alpha}{|\tau|}\sum_{i \in \tau_1}\va_i \va_i^T (\xk - \xopt) - \frac{\alpha}{|\tau|}\sum_{i \in \tau_2}(\va_i^T \xk - b_i) \va_i}^2 \\
    &=\norm{X - Y}^2 = \norm{X}^2 - 2\left\langle X, Y \right\rangle + \norm{Y}^2, 
\end{align*}
where $X := \left(\mI - \frac{\alpha}{|\tau|}\mA_{\tau_1}^T \mA_{\tau_1}\right)\ek$ with $\ek = \xk - \xopt$, and $Y := \frac{\alpha}{|\tau|}\sum\limits_{i \in \tau_2}(\va_i^T \xk - b_i)\va_i$. We proceed to analyze these three terms individually. 

\bigskip

\textbf{Term 1:} For the first uncorrupted term $\|X\|^2$, we pursue an analysis similar to that of \cite{Necoara2019FasterRB}, and let $\mW_{\tau_1} := \mA_{\tau_1}^\top\mA_{\tau_1}$. We then have:
\begin{align}
   \|X\|^2 = \norm{\left(\mI - \frac{\alpha}{|\tau|}\mW_{\tau_1}\right)\ek}^2 &= \ek^\top \left(\mI - \frac{\alpha}{|\tau|}\mW_{\tau_1}\right)^2\ek \nonumber\\
    &= \ek^\top \left(\mI - 2\frac{\alpha}{|\tau|}\mW_{\tau_1} + \frac{\alpha^2}{|\tau|^2}\mW_{\tau_1}^2\right)\ek. \label{term1}
\end{align}
Then, we estimate under the positive semi-definite (Loewner) ordering: $$\smin \leq \sigma_{\text{min}}^2(\mA_{\tau_1}) \le \mW_{\tau_1}  \leq \lambda_{\text{max}}(\mW_{\tau_1}) = \sigma_{\text{max}}^2(\mA_{\tau_1}) \leq \sigma_{\text{max}}^2,$$ 
recalling that $\sigma_{\text{max}}^2 := \sigma_{\text{max}}^2(\mA)$, and the smallest restricted singular values $\smin$ is defined as per $\eqref{smin_def}$. Furthermore, to ensure a decrease in norm at each iteration, we require $2\alpha/|\tau|  - \alpha^2\sigma_{\text{max}}^2/|\tau|^2 \geq 0$. This is a less restrictive bound than the later condition on $\alpha$ for convergence in \cref{eq:alphacond}, so we proceed. From \label{term1} we have
\begin{align*}
       \|X\|^2 &\leq \ek^\top \left(\mI - 2\frac{\alpha}{|\tau|}\mW_{\tau_1} + \frac{\alpha^2}{|\tau|^2}\sigma_{\text{max}}^2\mW_{\tau_1}\right)\ek \\
    &\leq \left(1 - \left(\frac{2\alpha}{|\tau|} - \frac{\alpha^2}{|\tau|^2}\sigma_{\text{max}}^2\right)\sigma_{\text{min}}^2(\mA_{\tau_1})\right)\norm{\ek}^2 \\
    &\leq \left(1 - \left(\frac{2\alpha}{qm} - \frac{\alpha^2}{q^2 m^2}\sigma_{\text{max}}^2\right)\smin\right)\norm{\ek}^2.
\end{align*}

\bigskip

\textbf{Term 2:}
For the scalar product between uncorrupted and accepted but corrupted parts $\langle X, Y\rangle$, we make use of \cref{lem:innerprodestimate}. We have
\begin{align*}
    \langle X, Y\rangle &\leq \frac{2\alpha}{|\tau|}\sum_{i \in \tau_2}\left|\left\langle \left(\mI - \frac{\alpha}{|\tau|}\mW_{\tau_1}\right)\ek, (\va_i^\top \xk - b_i)\va_i\right\rangle \right| \\
    &\leq \frac{2\alpha Q_q(\xk)}{|\tau|}\sqrt{|\tau_2|}\sigma_{\text{max}}\norm{\left(\mI - \frac{\alpha}{|\tau|}\mW_{\tau_1}\right)\ek} \\
    &\leq \frac{2\alpha Q_q(\xk)\sqrt{|\tau_2|}}{|\tau|}\sigma_{\text{max}}\left(1 - \frac{\alpha}{|\tau|}\smin\right)\norm{\ek} \\
    &\leq \frac{2\alpha\sqrt{\beta}\sigma_{\text{max}}^2}{qm\sqrt{1-q-\beta}}\left(1 -\frac{\alpha}{qm}\smin\right)\norm{\ek}^2.
\end{align*}
Here we use that $|\tau| = qm$, $|\tau_2| \leq \beta m$, and estimate $Q_q(\xk)$ using \cref{lem:quantileestimate}.

\bigskip

\textbf{Term 3:} Lastly, we estimate the maximal total impact of the residual constrained corrupted equations, making use of both \cref{lem:quantileestimate} and \cref{lem:normsum}:
\begin{align*}
    \|Y\|^2 &= \norm{\frac{\alpha}{|\tau|}\sum_{i \in \tau_2}(\va_i^\top \xk - b_i)\va_i}^2 \\
    &= \frac{\alpha^2}{|\tau|^2}\norm{\sum_{i \in \tau_2}(\va_i^\top \xk - b_i)\va_i}^2 \\
    &\leq \frac{\alpha^2 Q_q(\xk)^2}{|\tau|^2}\norm{\sum_{i \in \tau_2} \va_i}^2\\
    &\leq \frac{\alpha^2 Q_q(\xk)^2}{|\tau|^2}\sigma_{\text{max}}^2 |\tau_2| \\
    &\leq \frac{\alpha^2 \beta \sigma_{\text{max}}^4}{q^2 m^2(1-q-\beta)}\norm{\ek}^2.
\end{align*}

\bigskip

\textbf{Bringing all three estimates together} yields
\begin{align}\label{rate}
    \norm{\ekpo}^2 &\leq \biggl[1 - \left(\frac{2\alpha}{qm} - \frac{\alpha^2}{q^2 m^2}\sigma_{\text{max}}^2\right)\smin + \frac{2\alpha \sqrt{\beta}\sigma_{\text{max}}^2}{qm\sqrt{1-q-\beta}}\left(1 - \frac{\alpha}{qm}\smin\right)\nonumber\\
    & \qquad+ \frac{\alpha^2 \beta \sigma_{\text{max}}^4}{q^2 m^2(1-q-\beta)}\biggr]\norm{\ek}^2 = \biggl(1 - c_1\alpha + c_2\alpha^2\biggr)\norm{\ek}^2,
\end{align}
where
\begin{equation}\label{c1}
c_1 = \frac{2 \smin}{qm} - \frac{2 \sqrt{\beta}\sigma_{\text{max}}^2}{qm\sqrt{1-q-\beta}},
\end{equation}
\begin{equation}\label{c2}
c_2 = \frac{\sigma_{\text{max}}^2\smin}{q^2 m^2} - \frac{2\sqrt{\beta}\sigma_{\text{max}}^2\smin}{q^2 m^2 \sqrt{1-q-\beta}} + \frac{\beta \sigma_{\text{max}}^4}{q^2 m^2(1-q-\beta)}.
\end{equation}

In order to achieve convergence we must have $c_1 > 0$. This is equivalent to
\begin{equation*}
    \frac{\sqrt{\beta}}{\sqrt{1-q-\beta}} < \frac{\smin}{\sigma_{\text{max}}^2},
\end{equation*}
which is reminiscent of the relative conditions imposed on $q$ and $\beta$ in \cref{thm:steiner}, though slightly relaxed. With this restriction, we then have convergence for all $\alpha$ such that 
\begin{equation}\label{eq:alphacond}
1 - c_1 \alpha + c_2 \alpha^2 < 1,
\end{equation}
equivalently, $\alpha \in (0, c_1/c_2)$, with an optimal choice of $\alpha := c_1/2c_2$. With this optimal choice, our per-iteration guarantee becomes

\begin{equation*}
     \norm{\ekpo}^2 \leq \left(1 - \frac{c_1^2}{4c_2}\right)\norm{\ek}^2.
\end{equation*}
Induction then yields the result.
\end{proof}
\begin{remark}[Non-optimal choices of $\alpha$]\label{approx_rate}
Note that since the convergence rate has quadratic dependence on $\alpha$ \eqref{rate}, taking $\alpha = \xi \alpha_{opt}$ with $\xi \in (0,2)$ results in the convergence rate that is $(\xi - \xi^2/2)$ times the ``optimal" convergence rate. This implies certain stability in the choice of $\alpha$: an approximation within a small constant factor does not change the dependence of the convergence rate on any characteristics of the matrix $\bf A$. Further, we focus on estimating the optimal step size $\alpha = \alpha_{opt}$.
\end{remark}

We proceed now to prove \cref{cor:mainthmcor}, in effect giving a simplification of the convergence rate derived in \cref{thm:mainthm}.

\begin{proof}[Proof of \cref{cor:mainthmcor}]

Given that for some $\epsilon \in (0,1)$,
\[
\frac{\sqrt{\beta}}{\sqrt{1-q-\beta}} = \epsilon\frac{\smin}{\sigma_{\text{max}}^2},
\]
we may simplify the expression for the rate via its components $c_1$ and $c_2$ given by \eqref{c1} and \eqref{c2}. Specifically, it simplifies to  $$c_1 = \frac{2(1-\epsilon)\smin}{qm}
$$
and
\begin{align*}
    c_2 &= \frac{\sigma_{\text{max}}^2\smin}{q^2 m^2} - \frac{2\sqrt{\beta}\sigma_{\text{max}}^2\smin}{q^2 m^2 \sqrt{1-q-\beta}} + \frac{\beta \sigma_{\text{max}}^4}{q^2 m^2(1-q-\beta)} \\
    &= \frac{\sigma_{\text{max}}^2 \smin}{q^2 m^2} - \frac{2 \epsilon \sigma_{q-\beta, \text{min}}^4}{q^2 m^2} + \frac{\epsilon^2 \sigma_{q-\beta, \text{min}}^4}{q^2 m^2} \\
    &= \frac{\smin}{q^2 m^2}\left(\sigma_{\text{max}}^2 - \epsilon(2-\epsilon)\smin\right).
\end{align*}
We can thus express the theoretical optimal step size as
\[
\alpha = \frac{c_1}{2c_2} = \frac{qm(1-\epsilon)}{\sigma_{\text{max}}^2 - \epsilon(2-\epsilon)\smin},
\]
and our guaranteed convergence rate as
\[
1 - \frac{c_1^2}{4c_2} = 1 - \frac{(1-\epsilon)^2 \smin}{\sigma_{\text{max}}^2 - \epsilon(2 - \epsilon)\smin}.
\]

\end{proof}

\begin{remark}
Our convergence rate in the general case is difficult to compare with the rate for QuantileRK given in \cref{thm:steiner}, as both expressions are complex and quite different. However, comparing leading terms one may show that our rate is $\mathcal{O}(\smin / \smax)$, and the rate found in \cref{thm:steiner} is $\mathcal{O}(\smin / m)$, showing our method yields a speedup by a factor of $m/\smax \ge 1$ (recall that the normalization of the rows ensures that $m = \|\mA\|_F^2 \ge \smax$). In the next section, we are able to make this more precise for the particular case that $\mA$ satisfies the random matrix heuristic given in \cref{def:subgausstype}.
\end{remark}
\subsection{Subgaussian Case}\label{sec:proof_subgaus}

In this section we take the point of view of \cite{haddock2020quantilebased}, namely that $\mA$ belongs to the class of random matrices described by \cref{def:subgausstype}. Within this setting, we will show that $\smin$ and $\smax$ are both on the order of $m/n$, giving rise to \cref{thm:subgcase}: i.e., the theoretical convergence rate in this case is not dependent on $m$ or $n$.

As mentioned in \cite{haddock2020quantilebased} and discussed in greater detail in \cite{Steinerberger2021QuantileBasedRK}, a standard example of a matrix satisfying \cref{def:subgausstype} is one whose rows have been sampled independently from the uniform distribution on the sphere. Alternatively, one may sample rows from the standard multivariate Gaussian distribution, and then normalize. The benefit of introducing this random matrix model is that the spectra of such matrices are well studied. In particular, it allows for a high probability uniform lower bound on the smallest singular values of uniform-sized submatrices of $\mA$: we state (\cite{haddock2020quantilebased}, Proposition 1) below.
\begin{prop}[\cite{haddock2020quantilebased}, Proposition 1]\label{prop:haddprop1}
Let $\delta \in (0,1]$ and let $\mA \in \mathbb{R}^{m \times n}$ satisfy \cref{def:subgausstype} with constants $D$ and $K$. Then there exist absolute constants $C_1, C_2 > 0$ such that if $\mA$ has large enough aspect ratio, namely,
\[
\frac{m}{n} > C_1 \frac{1}{\delta}\log\frac{DK}{\delta},
\]
then the following high probability uniform lower bound holds for the smallest singular values of all its row submatrices that have at least $\delta m$ rows.
\begin{equation*}
    \mathbb{P}\left(\inf_{\tau \subseteq [m], |\tau| \geq \delta m}\sigma_{\mathrm{min}}(\mA_\tau) \geq \frac{\delta^{3/2}}{24D}\sqrt{\frac{m}{n}}\right) \geq 1 - 3\exp(C_2 \delta m).
\end{equation*}
\end{prop}
Equipped with this result, taking $\delta = q-\beta$ gives the following bound on our key quantity of interest $\smin(\mA)$.
\begin{corollary}
Suppose that $\mA$ satisfies Assumptions 1 and 2, and let $C_1, C_2$ be the absolute constants arising from \cref{prop:haddprop1} upon taking $\delta = q-\beta$. If
\[
\frac{m}{n} > C_1 \frac{1}{q-\beta}\log \frac{DK}{q-\beta},
\]
then with probability at least $1 - 3\exp(-C_2(q-\beta)m)$,
\[
\smin \geq \frac{(q-\beta)^3}{(24D)^2}\frac{m}{n}.
\]
\end{corollary}

Furthermore, we have the following standard bound on $\smax(\mA)$ (see, e.g. \cite{vershynin_2018}, Theorem 4.6.1.):
\begin{theorem}\label{thm:boundonsmax}
Let $\mA \in \mathbb{R}^{m \times n}$ be a random matrix satisfying \cref{def:subgausstype} with constants $K, D$. Then \[
\smax \leq (1 + CK^2)\frac{m}{n}
\]
with probability at least $1 - 2\exp(-cm)$, for some absolute constants $C,c > 0$.
\end{theorem}

Equipped with these, we may conclude \cref{thm:subgcase} directly from Corollary~\ref{cor:mainthmcor}:
\begin{proof}[Proof of \cref{thm:subgcase}]
The restriction on $\beta$ given in \cref{eq:betacondition}, and the optimal step size $\alpha$ given in \cref{eq:subgalpha}, follow immediately from plugging the estimates on $\smin, \smax$ (given in \cref{prop:haddprop1} and \cref{thm:boundonsmax} respectively) into the corresponding restriction and optimal step size formulae found in \cref{cor:mainthmcor}.

For the convergence result, we may similarly apply estimates of $\smax$ and $\smin$ to the convergence guarantee in \cref{cor:mainthmcor}. Using these, with probability at least $1 - 2\exp(-cm) - 3\exp(-C_2(q-\beta)m) \geq 1 - c_3\exp(-c_q m)$ we have that 
\begin{align*}
1 - \frac{(1-\epsilon)^2 \smin}{\smax - \epsilon(2 - \epsilon)\smin} &\leq 1 - \frac{(1-\epsilon)^2 \smin}{\smax} \\
& \leq 1 - \frac{(1-\epsilon)^2\frac{(q-\beta)^3 m}{(24D)^2 n}}{(1+CK^2) \frac{m}{n}} \\
&= 1 - \frac{(1-\epsilon)^2 (q-\beta)^3}{(1+CK^2)(24D)^2} \\
&= 1 - C_q.
\end{align*}
Now, $c_q$ and $C_q$ are absolute constants depending only on $q$, so we have that for any $k$,
\[
\norm{\xk - \xopt}^2 \leq (1 - C_q)^k \norm{\xinit - \xopt}^2,
\]
as claimed.

\end{proof}

\begin{remark}\cref{thm:subgcase} shows that in the subgaussian setting, QuantileABK enjoys a speedup by a factor of $n$ over QuantileRK (recall \cref{thm:qrkconv}). This is, heuristically, due to the fact that averaging the projections onto many rows yields a direction vector that points more directly towards $\xopt$ than any single projection. Hence, one may take a much larger step size, of order $n$ in the subgaussian setting. The convergence rate then enjoys a corresponding increase of the same order. We note that the optimal step size is likely closely related to the coherence of the matrix, in that larger step sizes may be used for matrices with nearly orthogonal rows (such as those of subgaussian-type). This is because the coherence, in some sense, determines how much information about the location of $\xopt$ can be obtained from a block of rows. We explore and comment on this phenomenon further in \cref{sec:experiments}.
\end{remark}

\section{Experimental Results}
\label{sec:experiments}

We divide our experiments into two main sections. We first present results using our method as presented in \cref{alg:QBRK}, including determining the optimal step size, exploring robustness with respect to the quantile parameter, and comparing performance with QuantileRK. We then perform an analysis of a variation of our method, in which only a subset of rows are taken at each iteration and used for computing the quantile and averaged direction vector. This reduces computational cost (at least when it is not possible to compute the full residual in parallel), but yields potentially slower per-iteration convergence. We explore this trade-off experimentally, and believe that our theoretical results may be extended to this method for sufficiently large sample sizes, but leave such theory to future work. Lastly, we also include a demonstration of how a projective block method may not converge in the sparse corruption setting.

\subsection{Results without subsampling}

We begin with our method as presented in \cref{alg:QBRK}. We perform experiments on systems lying on two geometrical extremes: ``Gaussian" systems, where the entries of each row are sampled i.i.d. $\mathrm{N}(0,1)$ and then each row is normalized; and "coherent" systems, where the entries of each row are sampled i.i.d. $\mathrm{Uniform}(0,1)$ and then each row is normalized. These choices are motivated by the fact that the performance of row projection methods such as ours depends heavily on the geometry of the system, in particular the coherence (that is, the pairwise inner products of rows). Gaussian systems are typically highly incoherent, whilst our coherent construction produces highly coherent systems.

We use $m=10000$ rows in all experiments. $\mA$ is constructed to be either "Gaussian" or "coherent" as described, and then $\xopt \in \mathbb{R}^{n}$ is constructed at random with $\mathrm{N}(0,1)$ entries. We then let $\vb = \mA \xopt$. Corruptions are placed uniformly at random and are taken to be of size $\mathrm{Uniform}(-100,100)$, which is large relative to the magnitude of the entries of $\vb$. Other parameters will be specified for each experiment.

Prior to giving plots showing convergence directly, we first conduct experiments to find the optimal choice of $\alpha$ and $q$, and then use these optimal choices for convergence plots and comparisons with QuantileRK.

\subsubsection{Optimal Step Sizes}

We begin by determining the optimal choice of step size $\alpha$ for the systems with $\mA \in \mathbb{R}^{10000 \times n}$, where $n \in \{10, 50, 100, 200, 500\}$. We take $\beta = 0.2$ and $q = 0.7$ and plot the relative error after 10 iterations, $\norm{\xk - \xopt}/\norm{\xinit - \xopt}$, versus $\alpha$. Since convergence is approximately linear, it suffices to run the method for only a few iterations to determine the optimal parameter. For Gaussian systems the optimal step size appears to scale with the number of columns $n$, and we present a scaled $x$-axis to highlight this. The optimal step size is around $1.6n$ to $1.8n$ for each $n$. For coherent systems, however, the step size does not scale in this fashion, and the optimal step size is approximately $2$ for all $n$. This corresponds to the heuristic that more information about the location of $\xopt$ is obtained when rows are more incoherent, and thus a larger step size may be taken. Note that a relative error greater than $1$ indicates that the method will diverge: for Gaussian systems this happens for step sizes roughly larger than $3n$, and for coherent systems divergence occurs for step sizes roughly larger than $2.5$.

\begin{figure}[h]
    \begin{subfigure}[b]{0.49\textwidth}
    \centering
    \includegraphics[width=\linewidth]{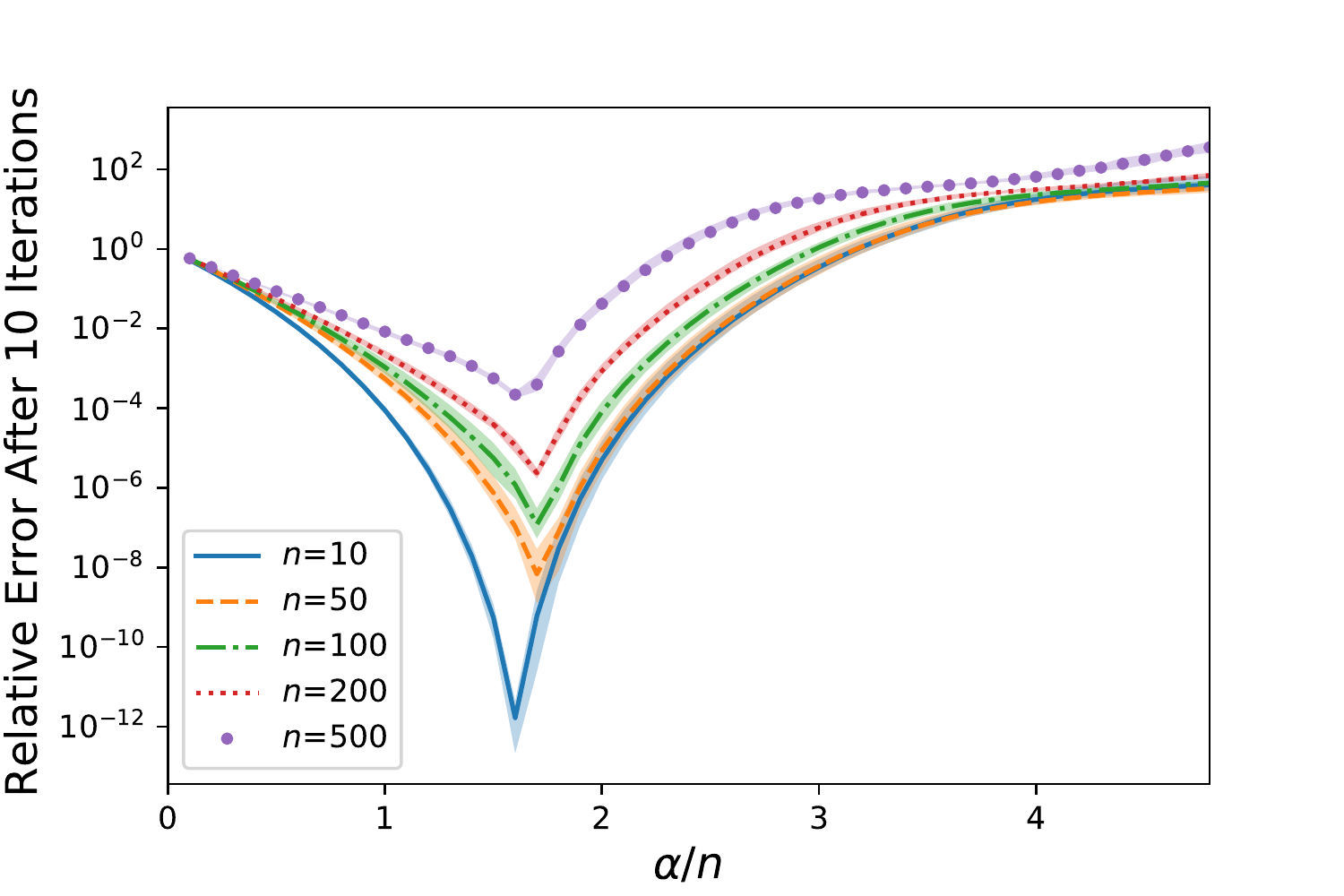}
    \caption{Gaussian systems}
    \label{fig:gauss_stepsize}
    \end{subfigure}
    \begin{subfigure}[b]{0.49\textwidth}
    \centering
    \includegraphics[width=\linewidth]{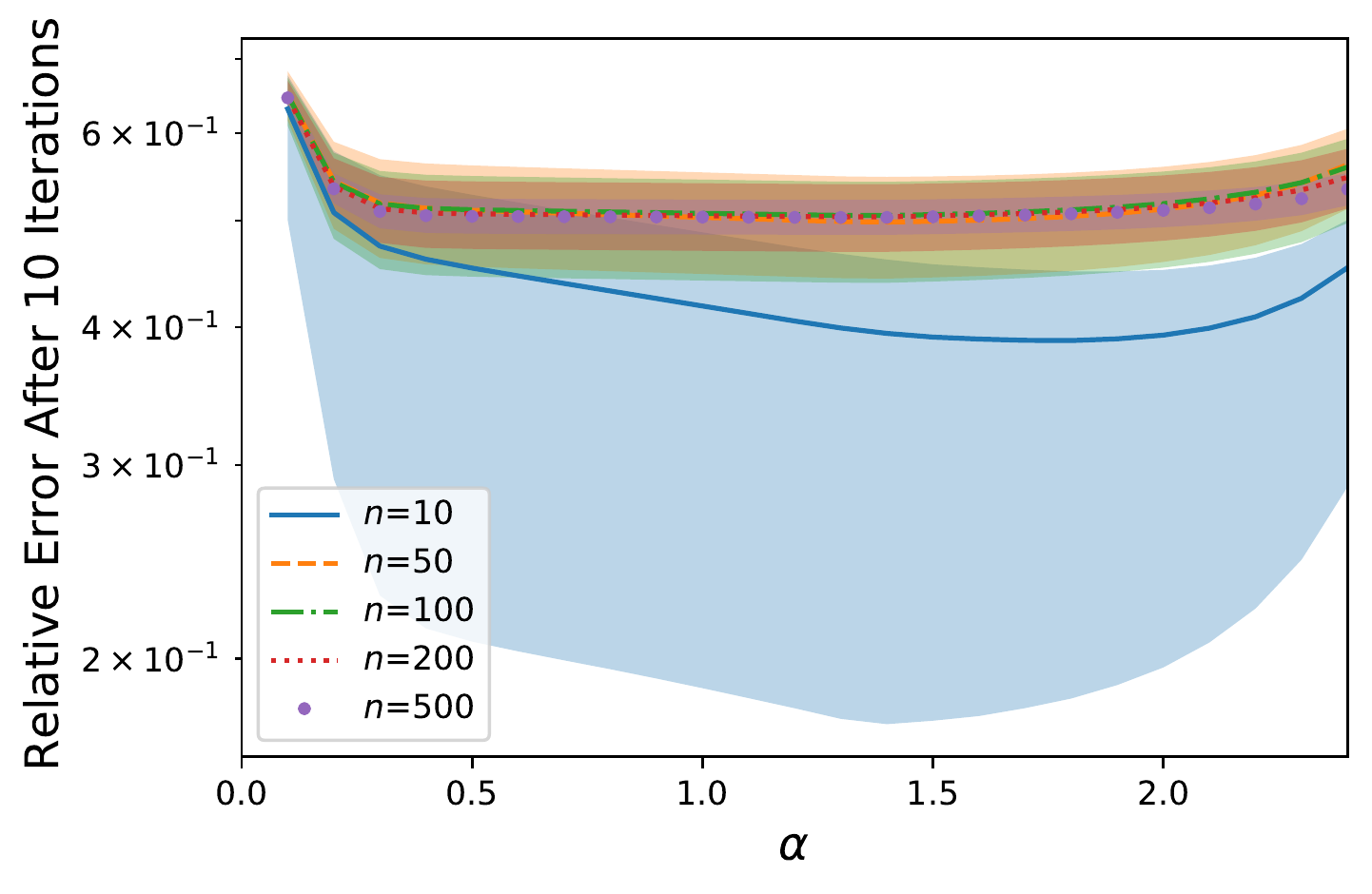}
    \caption{Coherent systems}
    \label{fig:coherent_stepsize}
    \end{subfigure}
    \caption{Relative error after 10 iterations of QuantileABK applied to $10000 \times n$ Gaussian and coherent systems versus step size, for different numbers of columns $n$. Note that in (a) the $x$-axis is scaled.}
    \label{fig:stepsizes}
\end{figure}

\subsubsection{Optimal choice of $q$}

The relative conditions imposed on $q$, $\beta$ in \cref{thm:mainthm} are strict, in the sense that $q$ must in general be much smaller than $1 - \beta$. However, we are able to show in practice that the method is robust even for $q$ very close to $1-\beta$. This is beneficial as taking $q$ to be larger allows for uncorrupted rows with larger residual entries to be used in the averaged projection step, leading to larger movement towards $\xopt$ and consequently accelerated convergence. In \cref{fig:qvsbeta} we take $\mA \in \mathbb{R}^{10000 \times 100}$, $\beta \in \{0.1, 0.2, 0.3, 0.4, 0.5\}$, and vary $q \in (0,1)$. We plot the relative error after $10$ iterations of QuantileABK($q$) with the optimal step size found experimentally as in the previous subsection.

Our results indicate that in both extremes of system geometry, the method is highly robust to $q$ and $q$ may be taken very close to $1-\beta$ before convergence begins to slow or fail entirely. In practice, estimating $\beta$ precisely may be difficult, so one may be more conservative when choosing $q$.

\begin{figure}[H]
    \begin{subfigure}[b]{0.49\textwidth}
    \centering
    \includegraphics[width=\linewidth]{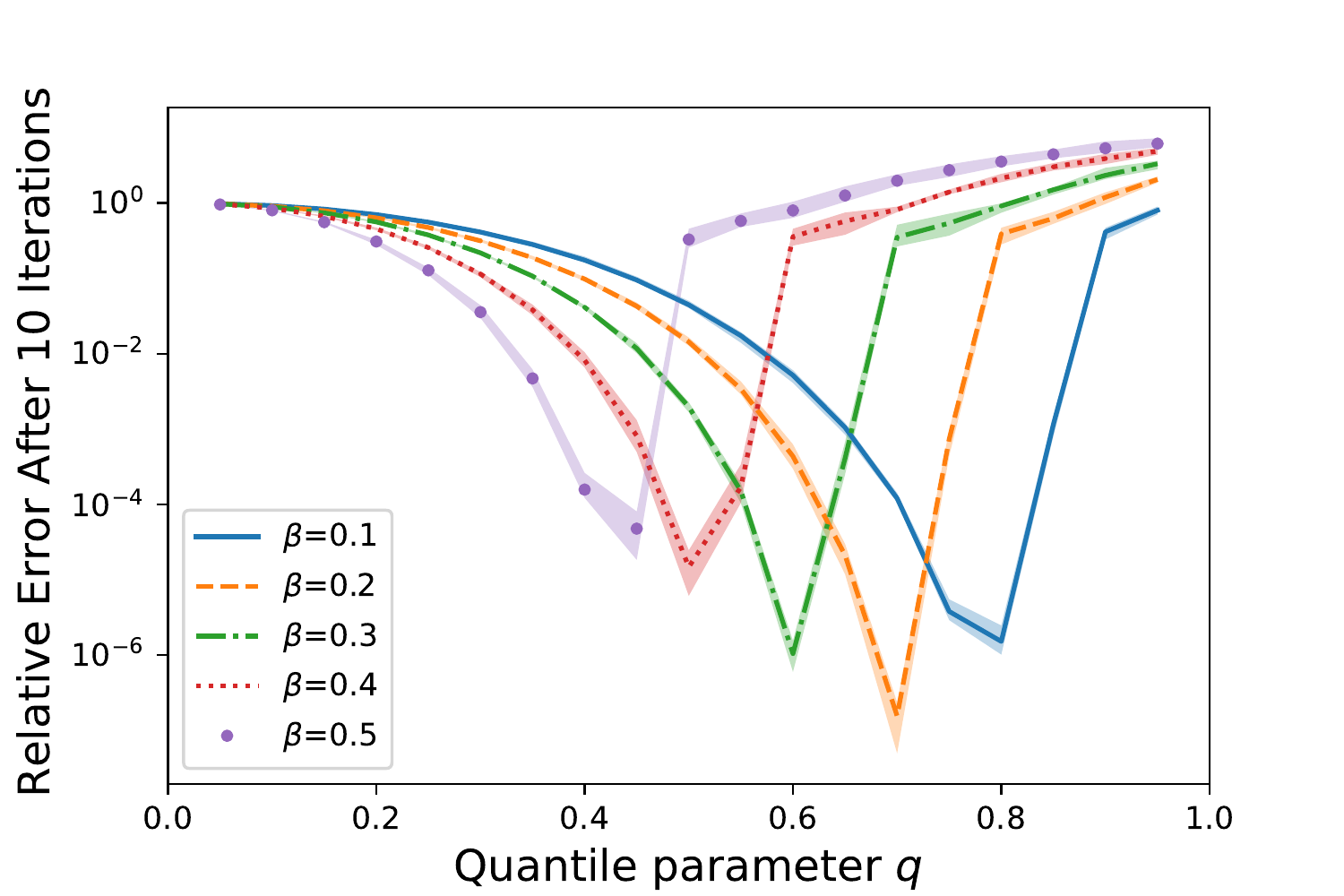}
    \caption{Gaussian system.}\label{fig:gauss_qvsbeta}
    \end{subfigure}
    \begin{subfigure}[b]{0.49\textwidth}
    \centering
    \includegraphics[width = \linewidth]{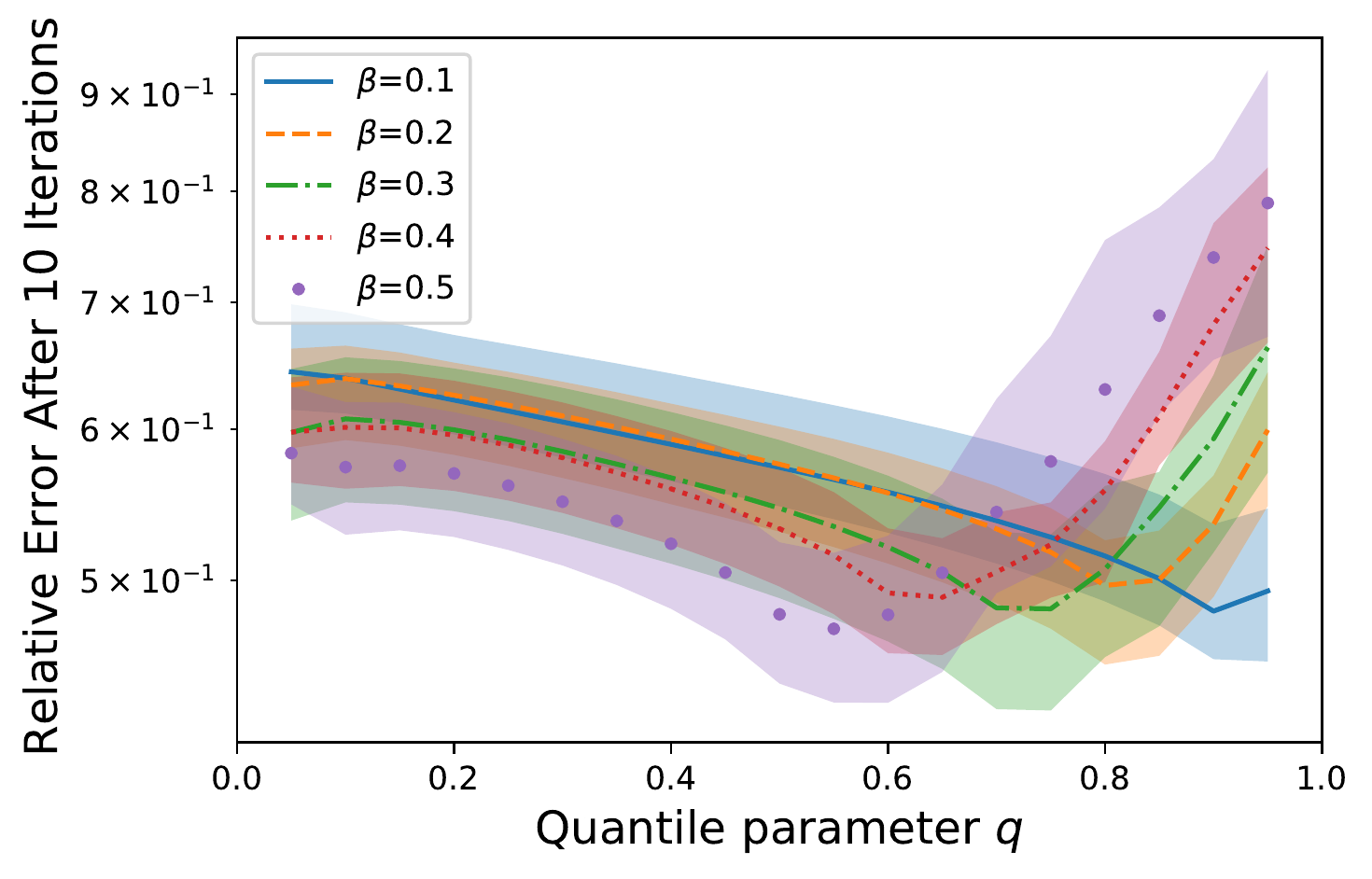}
    \caption{Coherent system.}\label{fig:coherent_qvsbeta}
    \end{subfigure}
    \caption{Relative error after 10 iterations of QuantileABK applied to $10000 \times 100$ Gaussian and coherent systems versus choice of quantile $q$, for a range of corruption rates $\beta$.}
    \label{fig:qvsbeta}
\end{figure}

\subsubsection{Acceleration over QuantileRK}

We compare QuantileABK to QuantileRK on $10000 \times 100$ systems. We take $\beta = 0.2$, $q = 0.7$ and perform 100 iterations of both methods. In \cref{fig:comparison} we plot the relative error of each method versus iteration, and also versus CPU time. It is clear that QuantileABK outperforms QuantileRK significantly in both the Gaussian and coherent settings, on both a per-iteration and temporal basis. We note that the plateau appearing in the Gaussian plots is due to floating point arithmetic limitations.

The plots for the coherent system show an initial sharp drop-off in the relative error before a more steady linear convergence. This is a reflection of an initial large movement when $\xinit$ is projected on the first selected hyperplane(s), and then subsequent small movements from further projections as the incident angles between hyperplanes are small. Note that both methods do converge when applied to the consistent system (but slowly, as suggested by Theorem 1.8, since coherency results in small values of $\sigma^2_{q - \beta, \min}$).

\begin{figure}[H]
  \begin{subfigure}[t]{.49\textwidth}
    \centering
    \includegraphics[width=\linewidth]{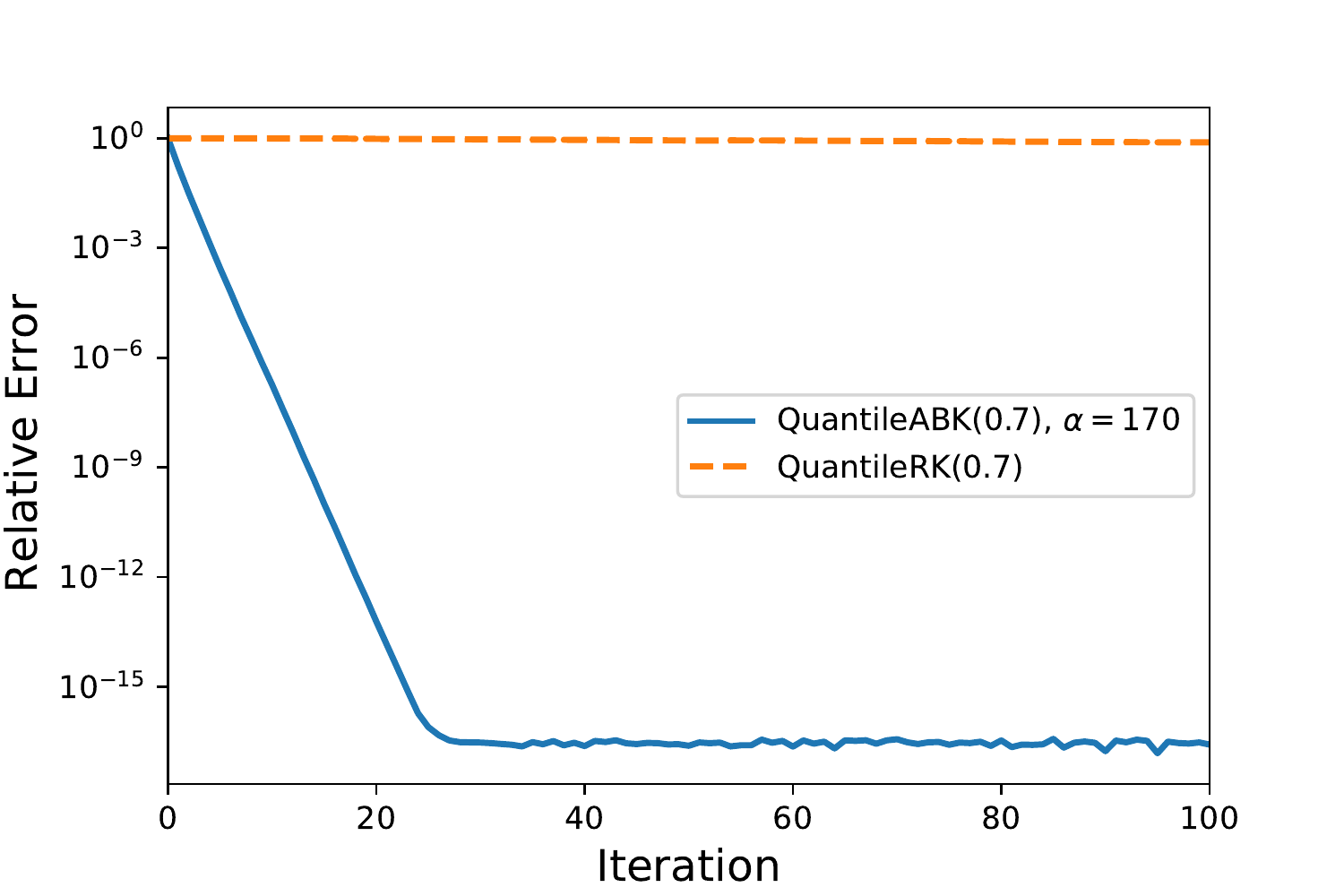}
    \caption{Relative error versus iteration, Gaussian system.}
  \end{subfigure}
  \hfill
  \begin{subfigure}[t]{.49\textwidth}
    \centering
    \includegraphics[width=\linewidth]{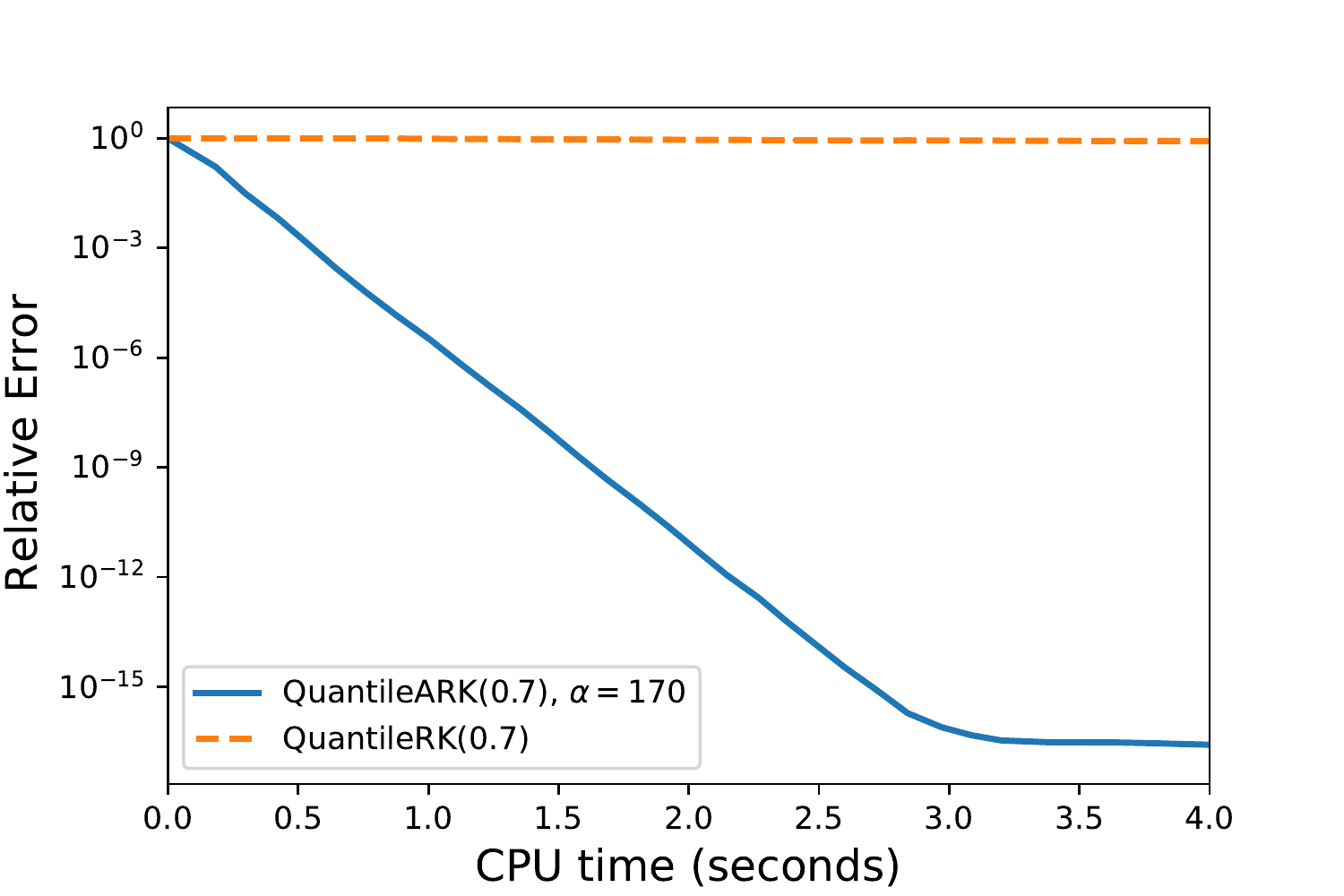}
    \caption{Relative error versus CPU time, Gaussian system.}
  \end{subfigure}

  \medskip

  \begin{subfigure}[t]{.49\textwidth}
    \centering
    \includegraphics[width=\linewidth]{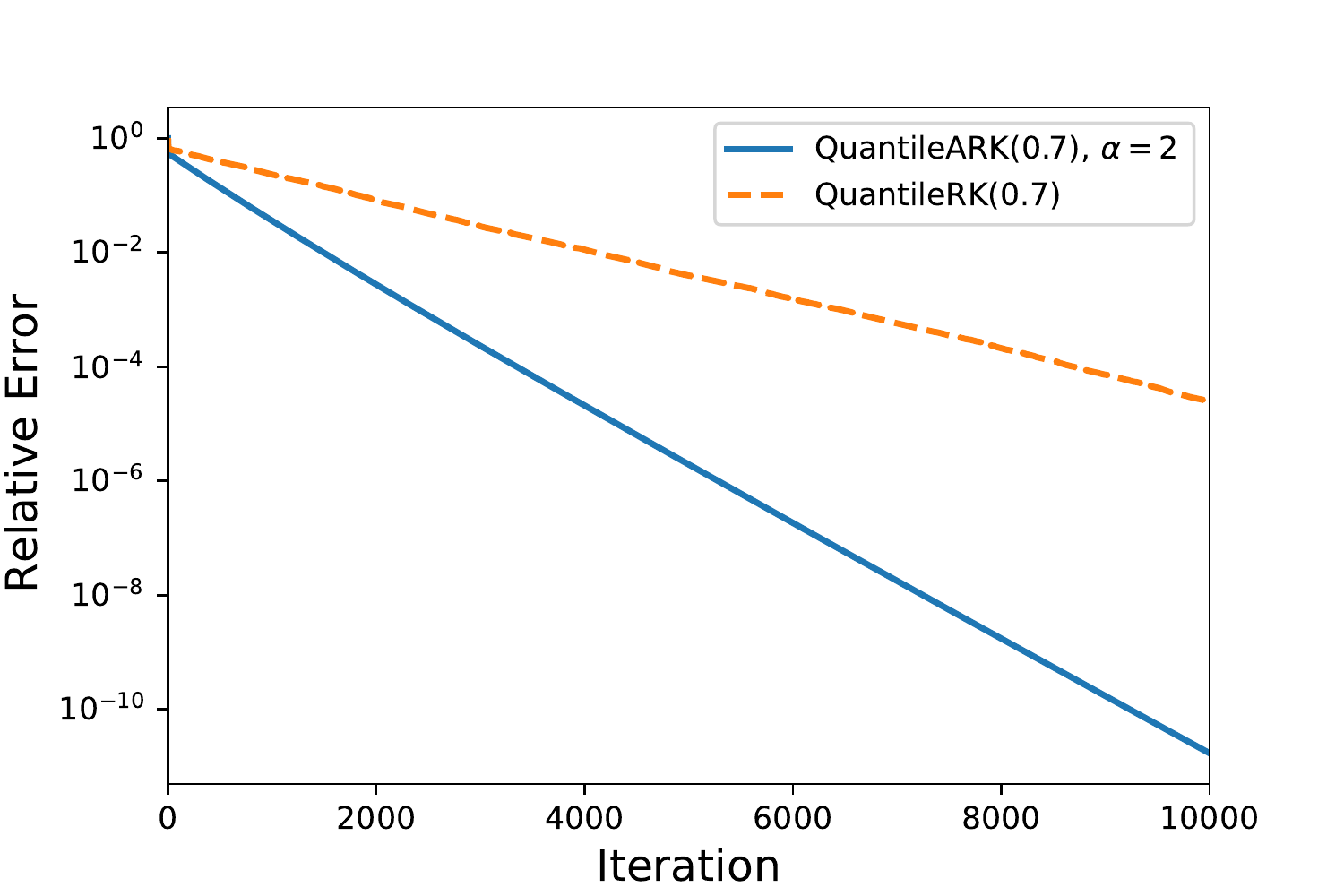}
    \caption{Relative error versus iteration, coherent system.}
  \end{subfigure}
  \hfill
  \begin{subfigure}[t]{.49\textwidth}
    \centering
    \includegraphics[width=\linewidth]{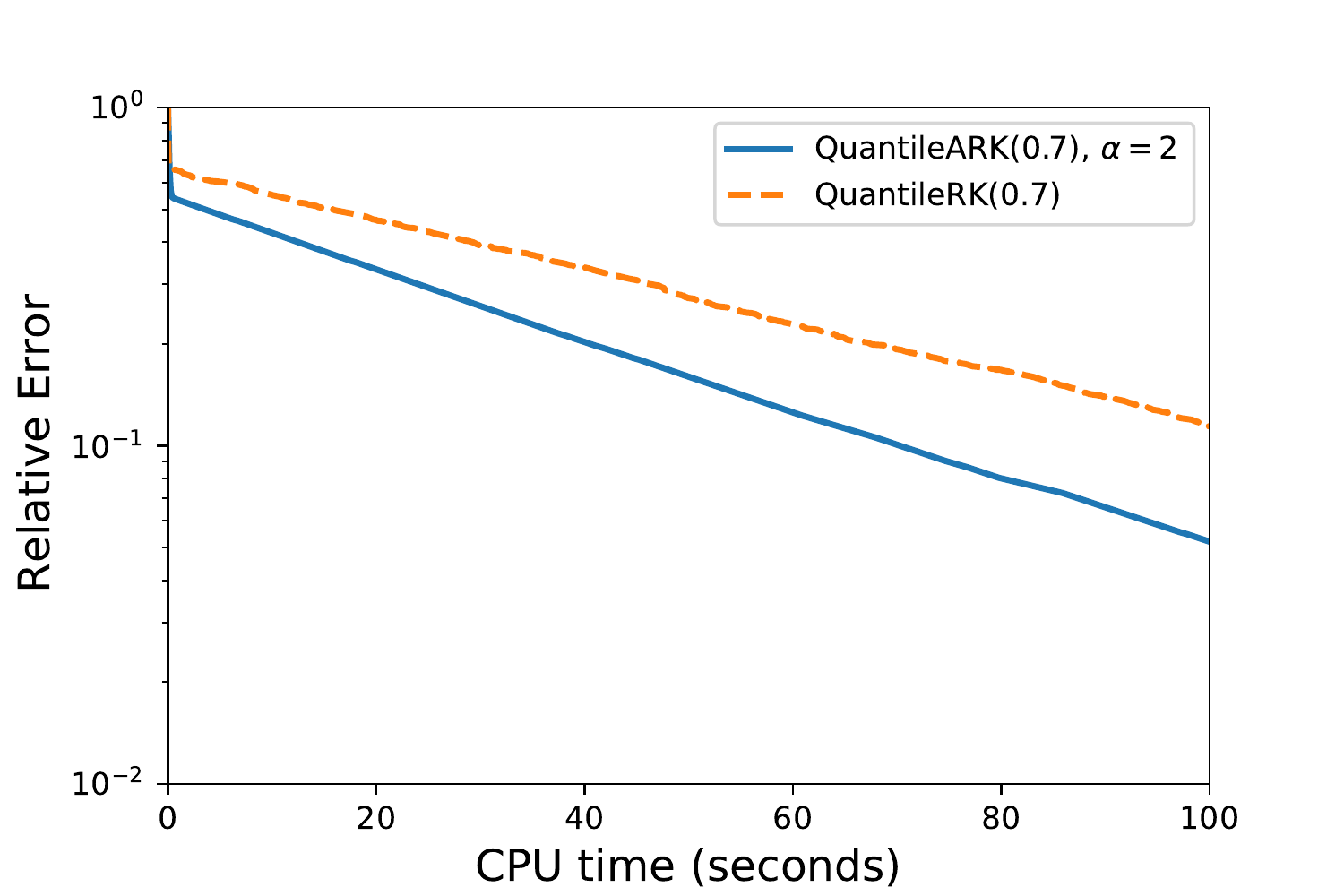}
    \caption{Relative error versus CPU time, coherent system.}
  \end{subfigure}
  \caption{Comparison of QuantileABK to QuantileRK on $10000 \times 100$ Gaussian and coherent systems, with quantile $q = 0.7$, corruption rate $\beta = 0.2$.}\label{fig:comparison}
\end{figure}

\subsection{Results with subsampling}

In this section, we perform experiments using a modification of our method, in which at each iteration only a sampled subset of the residual is computed. At each iteration, $t$ rows are sampled uniformly, the subresidual for that block of rows is computed, and its $q$-quantile is taken. An averaged projection step is then performed using the rows in this block with residual entries below the quantile, with step size $\alpha$. We call this method SampledQABK and give pseudocode in \cref{alg:SQBRK}. This methodology may be of interest when the full residual cannot be computed in parallel, as in this case subsampling can substantially reduce the computational cost (accompanied by a trade-off with the per-iteration convergence rate, as we will show). We note that our theoretical results require sampling the full residual, but we believe that this may be relaxed. 

\begin{algorithm}
	\caption{Sampled Quantile Averaged Block Kaczmarz}\label{alg:SQBRK}	\begin{algorithmic}[1]
		\Procedure{SampledQABK}{$\mA,\vb$, $N$, $q$, $t$, $\alpha$, $\xinit$}
		\For{$k = 1, 2, \ldots, N-1$}
		    \State Sample $i_1, \cdots, i_t \in [m]$ uniformly without replacement
		    \State Compute $Q_q(\vx_{k-1}) = q\textsuperscript{th}\text{ quantile of }\{|\va_i^T \vx_{k-1} - b_i| : i \in \{i_1, \cdots, i_t\}\}$
		    \State Set $\tau = \{i \in \{i_1, \cdots, i_t\} : |\va_i^T \vx_{k-1} - b_i| < Q_q(\vx_{k-1})\}$
    		\State Update $\xk = \vx_{k-1} - \frac{\alpha}{|\tau|}\sum_{i \in \tau}(\va_i^T \vx_{k-1} - b_i)\va_i$
		\EndFor
		\State \Return $\vx_N$
		\EndProcedure
	\end{algorithmic}
\end{algorithm}

Our experimental setup is the same as in the previous section: Gaussian and coherent systems are constructed in the same manner, corruptions are taken $\mathrm{Uniform}(-100,100)$ and placed uniformly at random, and other parameters will be specified for each experiment.

\subsubsection{Optimal Step Sizes}

We again begin by finding the optimal step size experimentally. We fix $10000 \times 100$ Gaussian and coherent systems and take $\beta = 0.2$, $q = 0.7$. We then take sample sizes $t \in \{100, 500, 1000, 5000\}$ and run SampledQABK for $10$ iterations for a range of step sizes $\alpha$, and present our results in \cref{fig:sampled_stepsizes}. As mentioned previously, the method converges linearly and so it is sufficient to run it for only a few iterations for comparison purposes. We note that the method was unstable or did not converge for $t < 100$, which is a consequence of the method being unable to accurately distinguish corrupted and uncorrupted rows given such a small sample.

\begin{figure}[H]
  \begin{subfigure}[t]{.49\textwidth}
    \centering
    \includegraphics[width=\linewidth]{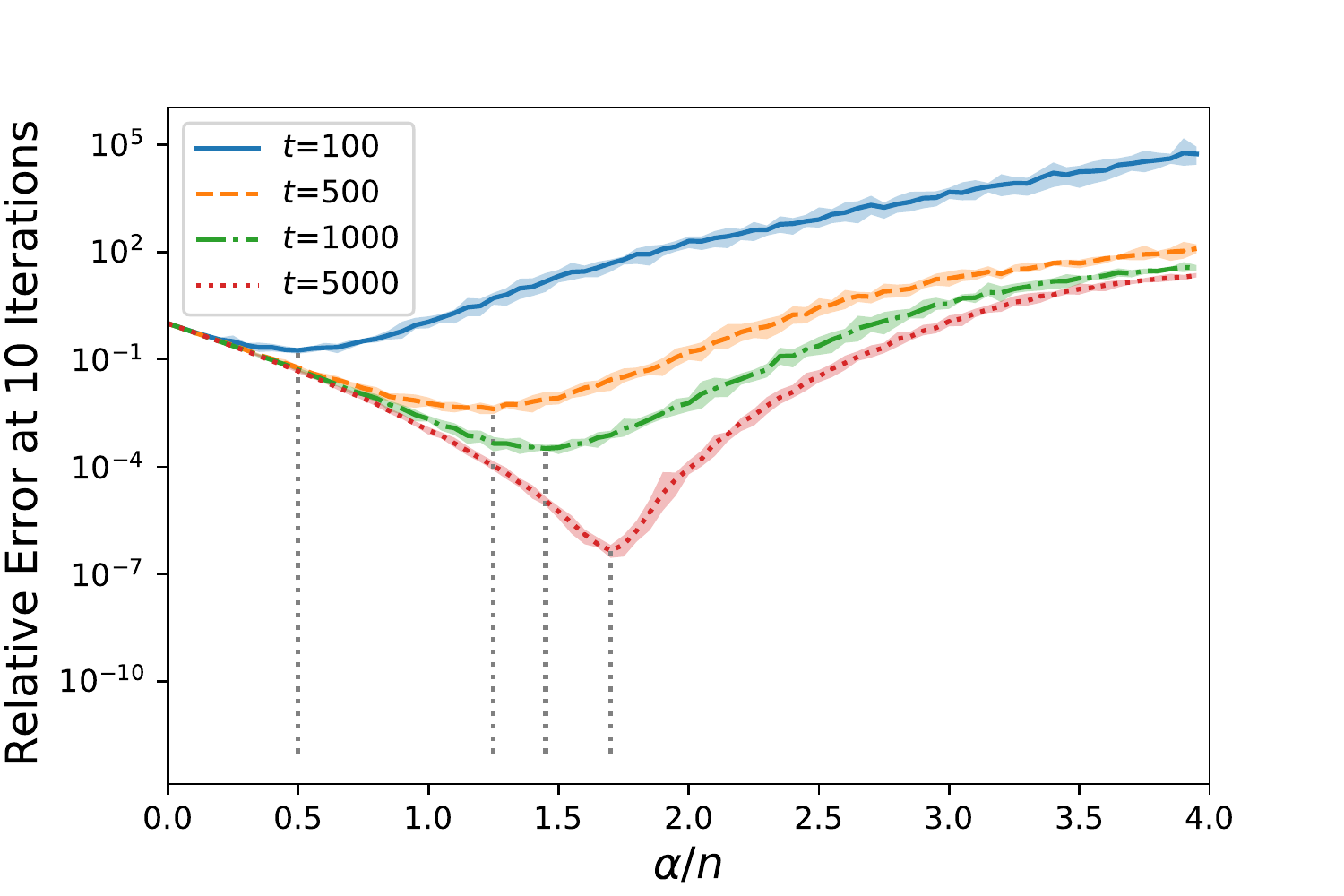}
    \caption{Gaussian system.}
  \end{subfigure}
  \hfill
  \begin{subfigure}[t]{.49\textwidth}
    \centering
    \includegraphics[width=\linewidth]{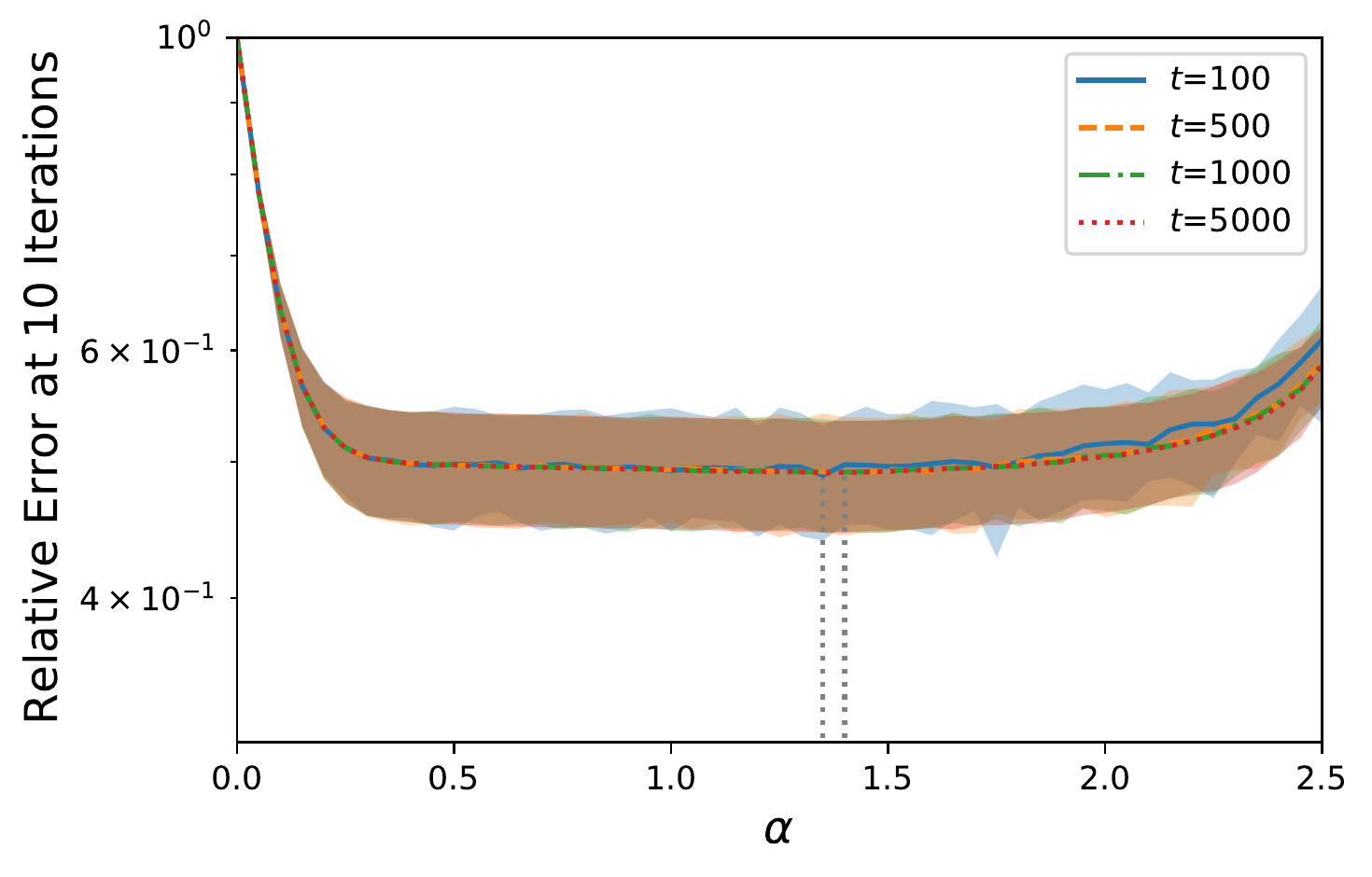}
    \caption{Coherent system.}
  \end{subfigure}
  \caption{Relative error after 10 iterations of SampledQABK versus step size $\alpha$ on $10000 \times 100$ Gaussian and coherent systems, for a range of sample sizes $t$, quantile parameter $q=0.7$, corruption rate $\beta = 0.2$. Note the scaling of the $x$-axis in (a).}\label{fig:sampled_stepsizes}
\end{figure}

We observe that for the Gaussian system, the optimal step size is again on the order of $n$, and that the method becomes more sensitive to the choice of step size as the sample size $t$ increases (that is, the 'valleys' at the optima become sharper). This may be explained by the heuristic that the amount of information that may be obtained from a block of rows (and in turn, the step size that may be taken) is limited by the rank of the matrix, which in the Gaussian case is $n$ almost surely. Hence, significant increases in sample size do not yield corresponding increases in the optimal step size.

For the coherent system, however, we see that the behavior is almost exactly the same across sample sizes. Again, the optimal choice of $\alpha$ is roughly constant, and there is little further information to leverage from taking larger sample sizes.

\subsubsection{Effect of Sample Size on Convergence}

We proceed to compare the convergence of SampledQABK with a variety of sample sizes. In \cref{fig:sampled_convergence} we include plots of the relative error versus iteration and versus CPU time, for both Gaussian and coherent systems. We compare sample sizes $t \in \{100, 500, 1000, 5000\}$, and use the optimal step sizes found in the previous section (for simplicity, we take $\alpha = 1.4$ for all sample sizes for the coherent system). We present our results in \cref{fig:sampled_convergence}.

For the Gaussian system, we see that taking a larger sample size greatly improves the per-iteration convergence. However, when computation time is taken into account, there is a clear trade-off: for example, $t=500$ converges much faster in terms of CPU time than $t = 5000$.

For the coherent system, we see that on a per-iteration basis there is essentially no difference between different sample sizes. The first iteration provides an initial jump, and then convergence proceeds much more slowly than the Gaussian case. Indeed, there is no trade-off between per-iteration convergence and computational cost in this case: subsampling greatly improves convergence over CPU time. We see that taking $t=100$ is significantly faster than any other sample size. In this case, $t$ should be taken as small as possible while still achieving convergence, and as noted previously we observed that convergence fails for $t < 100$.

\begin{figure}[H]
  \begin{subfigure}[t]{.49\textwidth}
    \centering
    \includegraphics[width=\linewidth]{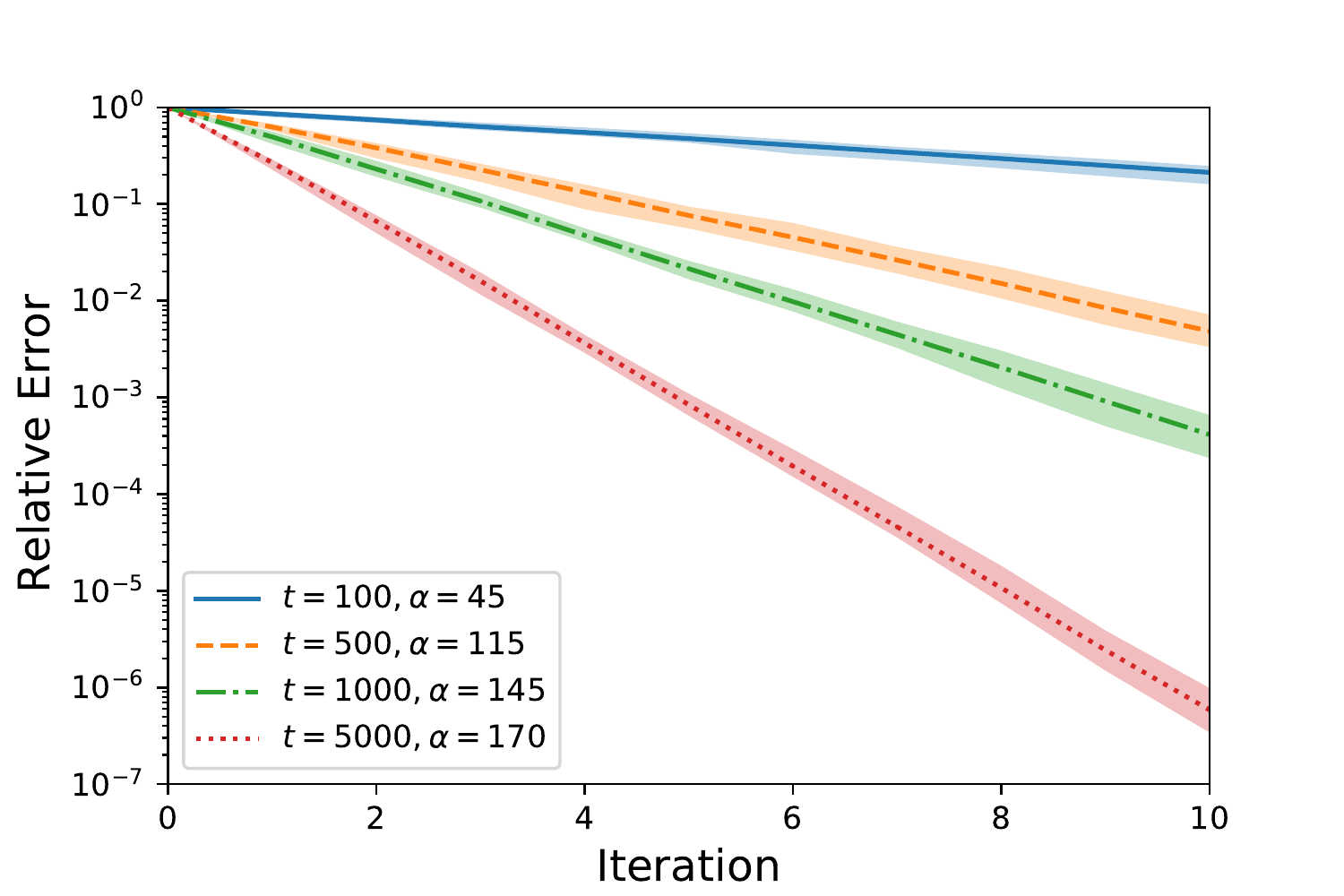}
    \caption{Relative error versus iteration, Gaussian system.}
  \end{subfigure}
  \hfill
  \begin{subfigure}[t]{.49\textwidth}
    \centering
    \includegraphics[width=\linewidth]{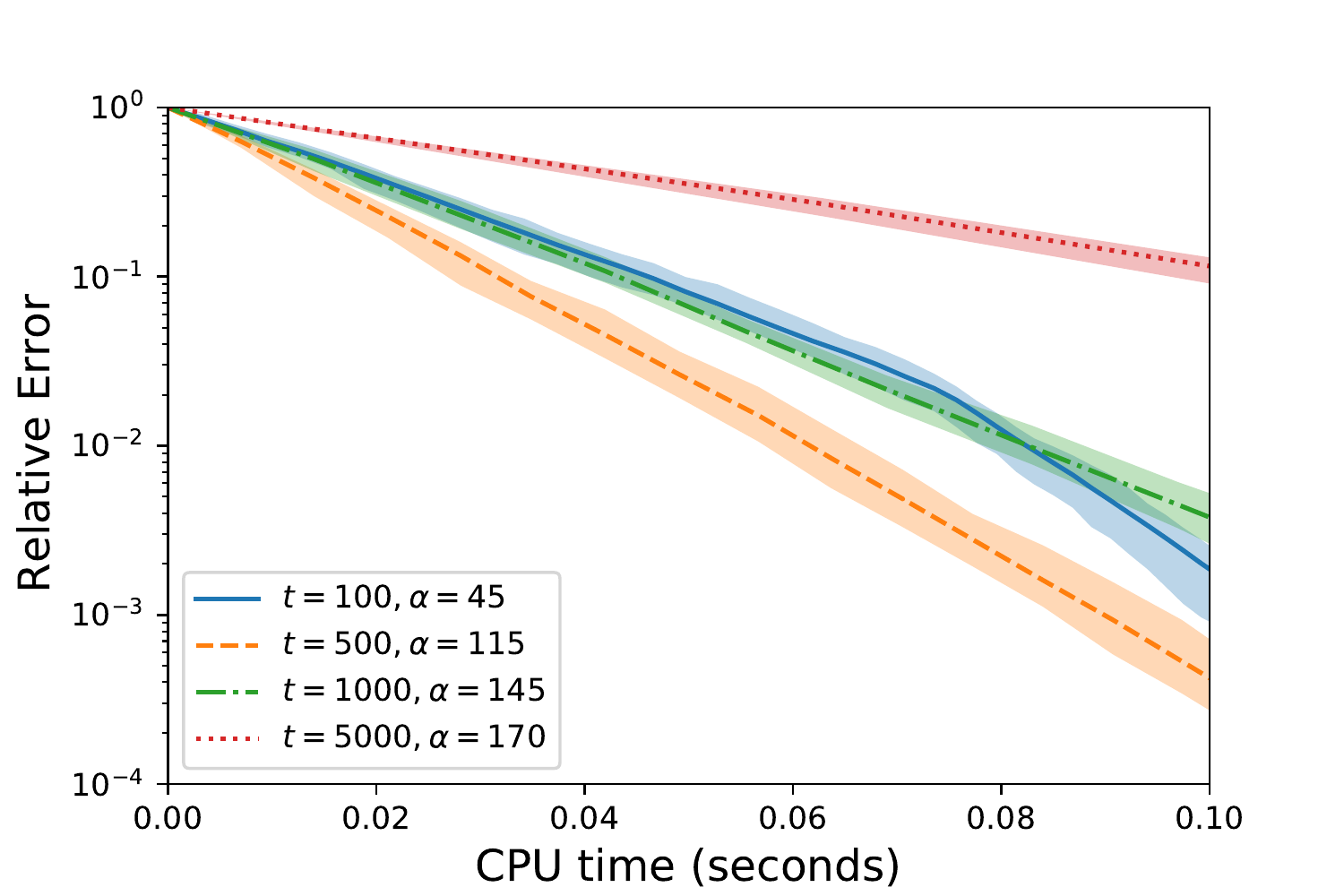}
    \caption{Relative error versus CPU time, Gaussian system.}
  \end{subfigure}

  \medskip

  \begin{subfigure}[t]{.49\textwidth}
    \centering
    \includegraphics[width=\linewidth]{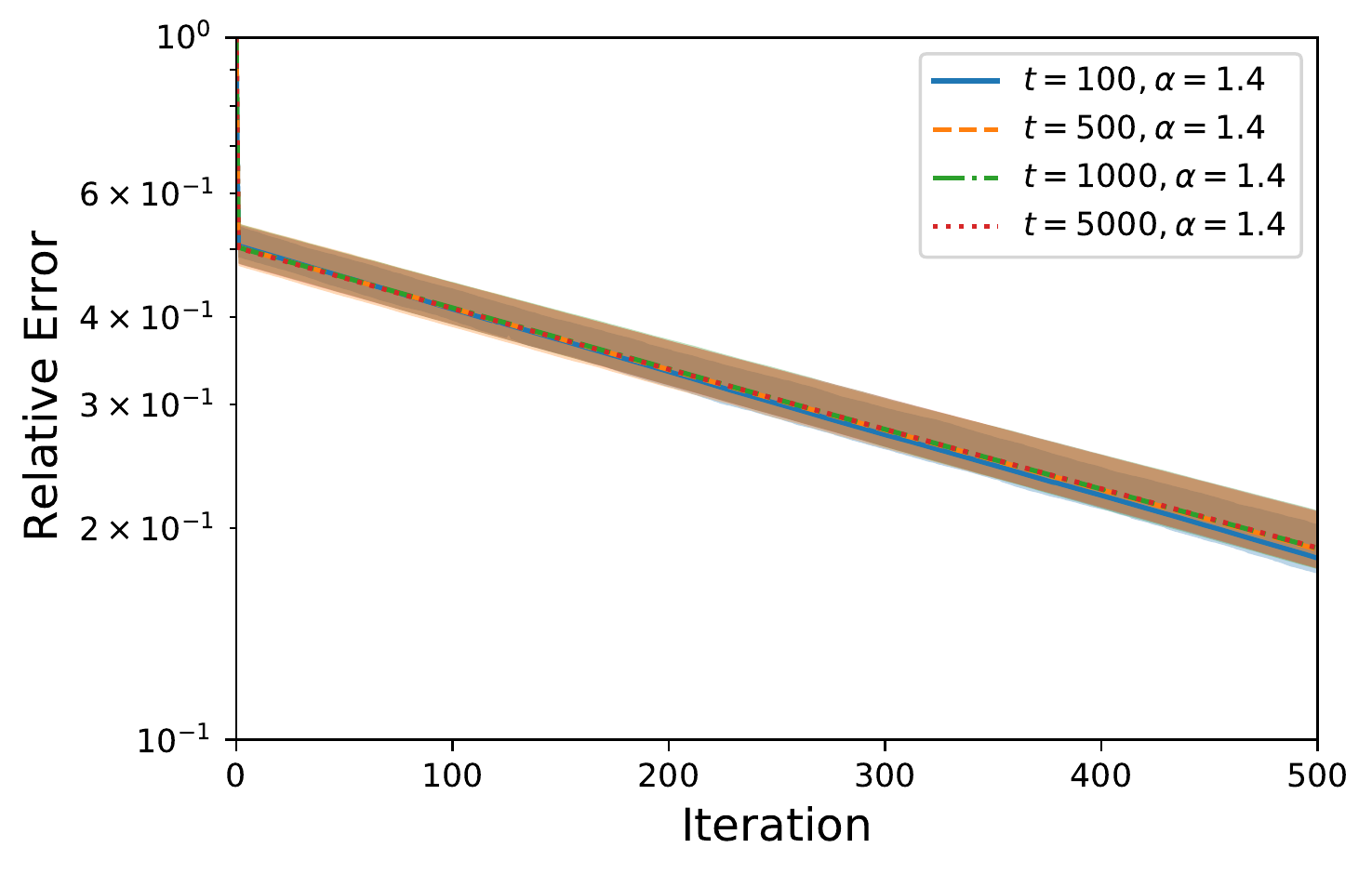}
    \caption{Relative error versus iteration, coherent system.}
  \end{subfigure}
  \hfill
  \begin{subfigure}[t]{.49\textwidth}
    \centering
    \includegraphics[width=\linewidth]{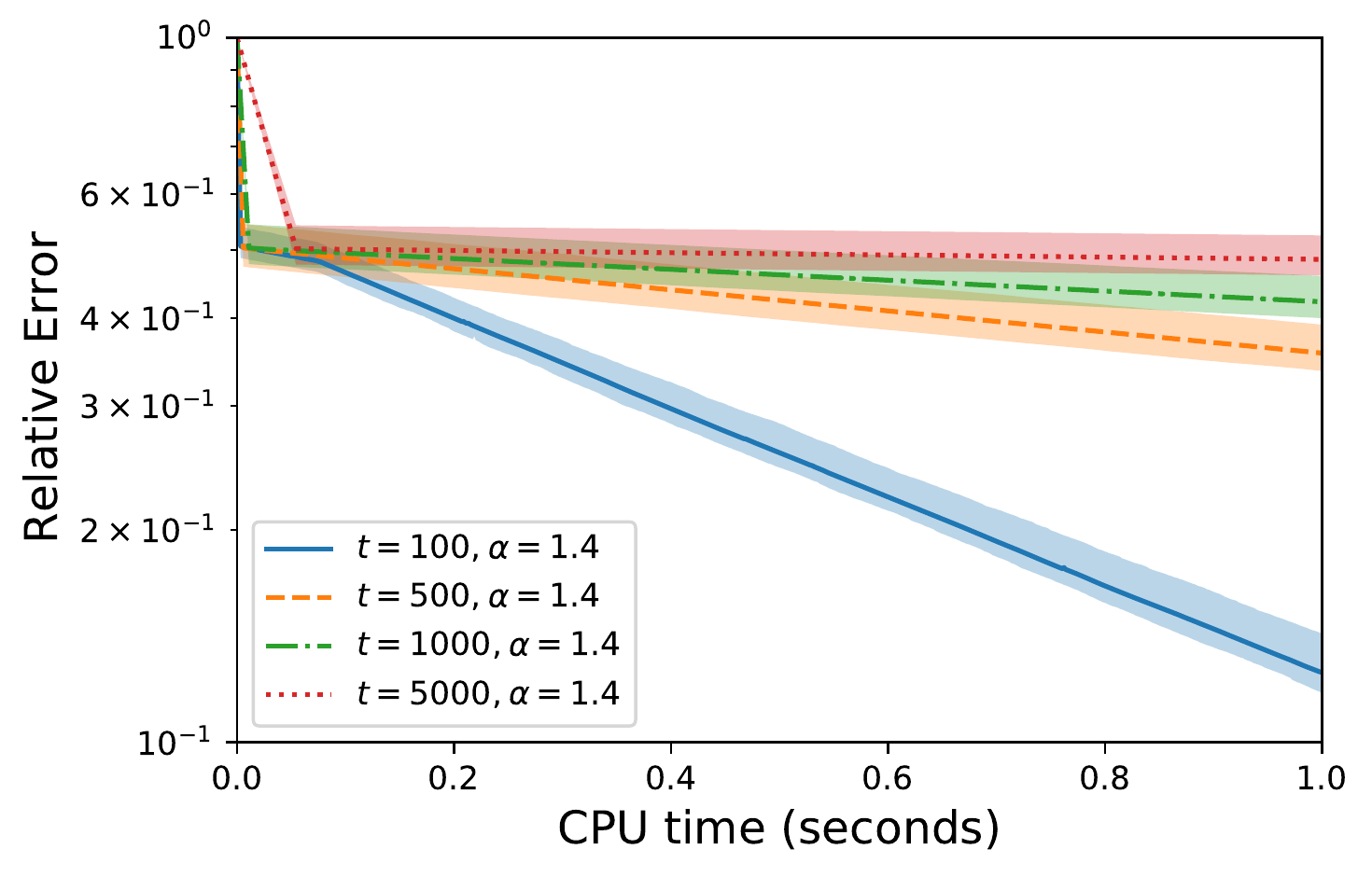}
    \caption{Relative error versus CPU time, coherent system.}
  \end{subfigure}
  \caption{Convergence of SampledQABK for a range of sample sizes $t$.}\label{fig:sampled_convergence}
\end{figure}

\subsection{Projective vs Averaged}\label{s:proj_vs_adv}
In this section, we give an experiment to support earlier discussion in \cref{s:block-k} regarding averaged versus projective block variants. We argued that in projective block methods, where iterates are projected onto the intersection of the hyperplanes corresponding to an entire block of rows, the presence of even a single corrupted row in each block can prevent convergence. To illustrate this we first give a natural quantile-based block Kaczmarz variant, QuantilePBK, in \cref{alg:QPBK}. Similar to QuantileABK, at each iteration, a quantile of the residual is taken, and then the previous iterate is projected onto the intersection of the hyperplanes of every row with residual entry beneath the quantile.

\begin{algorithm}[H]
	\caption{Quantile Projective Block Kaczmarz}\label{alg:QPBK}
	\begin{algorithmic}[1]
		\Procedure{QuantilePBK}{$\mA,\vb$, $N$, $q$, $\xinit$}
		\For{$k = 1, 2, \ldots, N-1$}
		    \State Compute $Q_q(\vx_{k-1}) = \text{q\textsuperscript{th} quantile of }\{|\va_i^T \vx_{k-1} - b_i| : i \in [m]\}$
		    \State Set $\tau = \{i \in [m] : |\va_i^T \vx_{k-1} - b_i| < Q_q(\vx_{k-1})\}$
		    \State Update {$\vx_k = \vx_{k-1} + \mA_{\tau}^{\dagger}(\vb_{\tau} - \mA_{\tau}\vx_{k-1})$}
		\EndFor
		\State \Return $\vx_N$
		\EndProcedure
	\end{algorithmic}
\end{algorithm}

 We construct an example to demonstrate how QuantilePBK may fail as follows. We construct a matrix $\mA \in \mathbb{R}^{1250 \times 100}$, where $1000$ rows are sampled by taking i.i.d. $N(0,1)$ entries and then normalizing, and where $250$ rows are identical copies of one further Gaussian row. A solution vector $\xopt$ is then constructed with i.i.d $N(0,1)$ entries. Then, the $250$ identical rows have their entries in $b$ corrupted in order to all equal $500$, given $250$ identical full rows in the system. Denoting one such row by $\va^\top \vx = 500$, the initial iterate $\xinit$ is then taken to be $\xinit = (500 - \va^\top\mathbf{1})\va^\top$, i.e., the projection of the vector of all ones $\mathbf{1}$ onto the hyperplane $\{ \va^\top \vx = 500 \}$. Under these choices, the iterates will always lie in this (corrupted) hyperplane, hence these rows will always have residual entry zero. This ensures they always pass the quantile test, and ensures that QuantilePBK cannot converge. We note that even taking projections on smaller sub-blocks of the accepted index set $\tau$ won't improve robustness, as long as each block contains one of the corrupted rows. So, in the worst case, even $250$ blocks of $5$ equations each can be such that the iterates never leave the corrupted hyperplane.

We perform QuantilePBK and QuantileABK on this system with $q = 0.7$, initial iterate as described above, and for QuantileABK a step size of $\alpha = 10$. In \cref{fig:projection_vs_averaged}, we plot the relative error versus iteration, and indeed observe as expected that QuantilePBK fails to converge, whilst QuantileABK continues to enjoy linear convergence even in this adversarial setting.

\begin{figure}[H]
    \centering
    \includegraphics[width=0.5\textwidth]{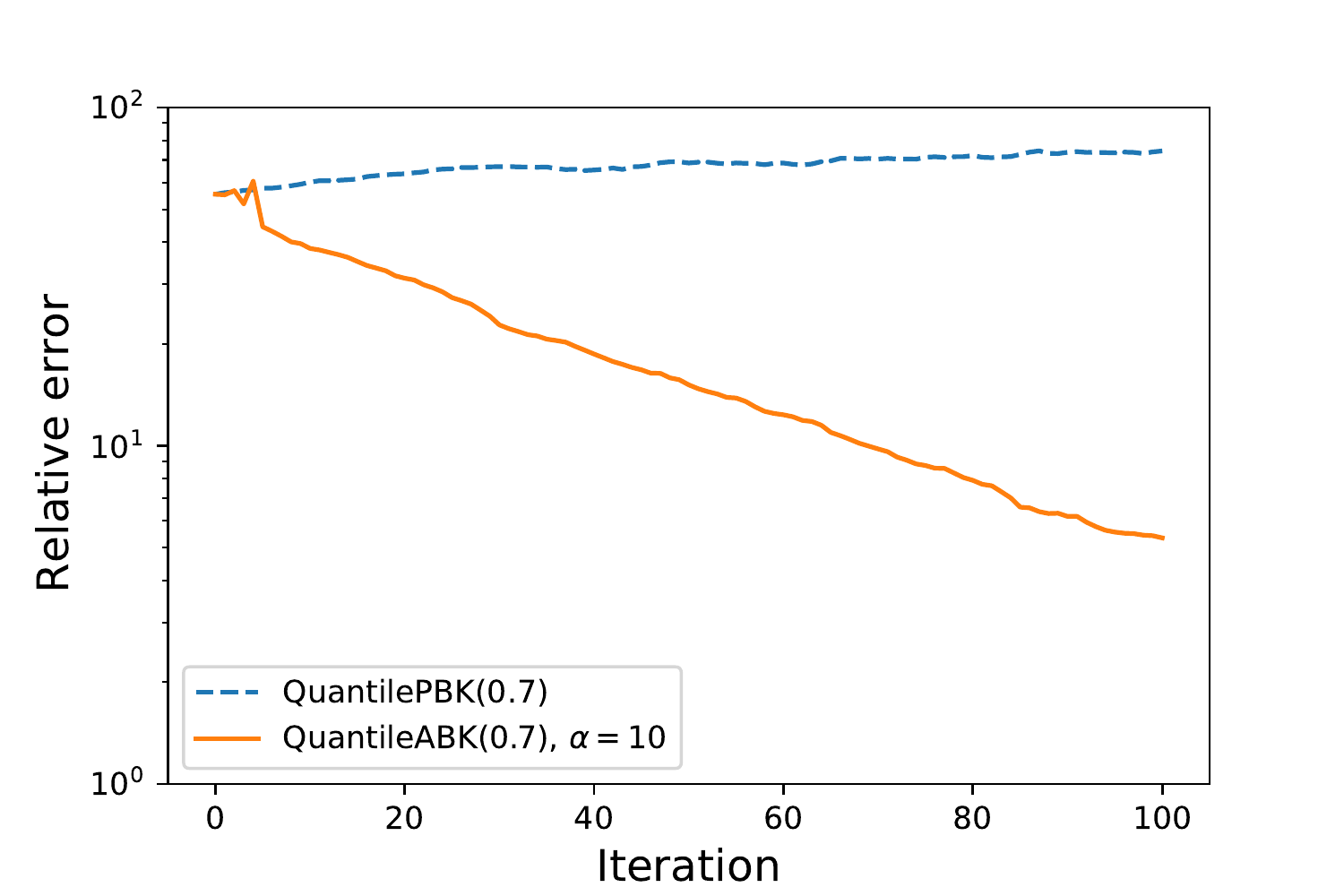}
    \caption{Relative Error versus iteration, QuantilePBK versus QuantileABK.}
    \label{fig:projection_vs_averaged}
\end{figure}

\section{Conclusion}\label{sec:conclusion}

In this work, we propose a novel method, QuantileABK, for solving large-scale systems of equations that suffer from arbitrarily large, but sparse, corruptions in the measurement vector. This sparse corruption model arises in a wide range of applications, including sensor networks, computerized tomography, and many problems in distributed computing, and finding methods that are able to detect and avoid these corruptions has been a popular recent problem.

Our method combines an averaged blocking technique, that has experienced recent popularity in related literature, with the use of a quantile of the residual at each iteration. This provides a large acceleration over the preeminent existing method for this setting, QuantileRK, by leveraging far more information from the computed residual at each iteration.

We prove that our method enjoys linear convergence under certain conditions on the quantile parameter $q$, and the fraction of corruption rates $\beta$, for all matrices such that the uniform smallest singular value over all row-submatrices with at least $(q-\beta)m$ rows is positive, i.e. $\smin > 0$. Notably, our results place no restriction on the size or (potentially adversarial) placement of corruptions. We show theoretically and experimentally that our method converges faster than QuantileRK. In particular by specializing to the case of a matrix of subgaussian-type, we are able to quantify this speed-up more precisely, and show that our method converges faster than QuantileRK by a factor of $n$ (the number of columns of the system). Whilst this speed up is per-iteration, both methods require computing the full residual at each iteration, so the per-iteration computational cost is of the same order.

Experimentally, we show that our method significantly outperforms QuantileRK, by iteration and by CPU time. We provide experiments on both geometric extremes (that is, matrices with nearly parallel rows and matrices with nearly orthogonal rows), and demonstrate the scaling behavior of the optimal step size in these cases, as well as direct performance comparisons to QuantileRK, in which the increase in convergence rate is clear. We also introduce a variant of our method that uses only a subsample of rows at each iteration, and provide step size and convergence results for a range of sample sizes.

As future work, we propose that there is still further information to be gained from the residual. In particular, we believe that historical residual information may be used to estimate the likelihood of a row being corrupted. That is, if one row's residual entries are continually greater than the quantile threshold, then that row is more likely to be corrupted than others. This could potentially then be used to reduce the number of corrupted rows that are deemed acceptable for projection at each iteration.

\section{Acknowledgements} The authors are grateful to Jackie Lok and anonymous reviewers for the valuable comments improving the presentation of the paper.

\printbibliography
\end{document}